\documentclass[12pt,twoside,a4paper]{article}
\usepackage{url}

\usepackage[utf8]{inputenc}
\usepackage[T1]{fontenc}
\usepackage{amsmath}
\usepackage{amsfonts}
\usepackage{amssymb}
\usepackage[numbers]{natbib}

\usepackage[bookmarksopen=false,breaklinks=true,%
      backref=page,pagebackref=true,plainpages=false,%
      hyperindex=true,pdfstartview=FitH,%
      pdfpagelabels=true,
      colorlinks=true,linkcolor=blue,%
      citecolor=red,urlcolor=red,
      colorlinks=true%
      ]%
   {hyperref}

\newtheorem{theorem}{Theorem}[section]
\newtheorem{lemma}[theorem]{Lemma}
\newtheorem{proposition}[theorem]{Proposition}

\newtheorem{remark}[theorem]{Remark}

\newtheorem{definition}[theorem]{Definition}

\newtheorem{example}[theorem]{Example}
\newtheorem{comput}[theorem]{Computational problem}

\newcommand{\junk}[1]{}
\def\={\hbox{\ \bf =}\ }

\newenvironment{proof}[1]{
\trivlist \item[\hskip \labelsep{\bf #1.}]\hskip 0pt\\}{\hfill 
\mbox{$\Box$}
\endtrivlist}
\makeatletter

\def\bar#1{\overline{#1}} 

\def \S{{\bf S}}
\def \V{{\bf V}}
\def \K{{\bf K}}
\def \L{{\bf L}}

\def \P{{\overline P}}
\def \x{{\overline x}}

\def\C{\mathfrak{C}}
\def\D{\mathfrak{D}}

\def \Inp{{\bf Input: }}
\def \Output{{\bf Output: }}

\def \Kac{{\K^{\rm ac}}}
\def \Kh{{\K^{\rm h}}}
\def \Vac{{\V^{\rm ac}}}
\def \GK{\Gamma_\K}

\def \GKac{\Gamma_{\Kac}}
\def \GKd{\Gamma_\K^{dh}}
\def \ResK{\overline{\K}}

\def \ResKac{\overline{\Kac}}

\def \U{\mathcal{U}}
\def \M{\mathcal{M}}
\def \MV{\M_{\V}}
\def \UV{\U_{\V}}
\def \UuV{\U^1_{\V}}
\def \MVac{\M_{\Vac}}

\def \GCD{\gcd}

\def \Z{\mathbb{Z}}
\def \N{\mathbb{N}}

\def \pgn{Newton poly\-gon }
\def \NPA{Newton Poly\-gon Algorithm }
\def \NPAz{Newton Poly\-gon Algorithm}

\def\impl{\mathrel{\Longrightarrow}}
\def\equi{\mathrel{\Longleftrightarrow}}

\def \noi {\noindent}
\def \ss {\smallskip}
\def \sni {\ss\noi}
\def \ms {\medskip}

\def\gen#1{\left\langle{#1}\right\rangle}

\title{Dynamic Computations Inside 
the Algebraic Closure of a Valued Field}
\author{Franz-Viktor Kuhlmann and 
Henri Lombardi and
and Herv\'e Perdry}

\date{2003}
\begin{document}

\maketitle

This paper appeared as

{{Kuhlmann, Franz-Viktor, Lombardi, Henri and Perdry,
              Herv{\'e}.
 {Dynamic computations inside the algebraic closure of a valued
              field},
p.\ 133-156 in the book {\it Valuation theory and its applications, {V}ol.\ {II}
              ({S}askatoon, {SK}, 1999)},
{Fields Inst.\ Commun.\ Vol.\ {33},}
 {Amer.\ Math.\ Soc., Providence, RI},
{(2003)},
}

\begin{abstract}
We explain how to compute in the algebraic closure of a
valued field. These computations heavily rely on the \NPAz. They are
made in the same spirit as the dynamic algebraic closure of a
field. They give a concrete content to the theorem saying that a
valued field does have an algebraically closed valued extension. The
algorithms created for that purpose can be used to perform an
effective quantifier elimination for algebraically closed valued
fields, which relies on a very natural geometric idea.
\end{abstract}

\medskip\noindent Key words: Valued fields, Quantifier elimination,
Constructive mathematics, Dynamical algebra

\medskip\noindent MSC 2000: {Primary: 12J10, 12J20, 12Y05, 03C10; Secondary: 13P05, 68W30}

\bigskip  \noindent 
Franz-Viktor Kuhlmann. Mathematical Sciences Group, University of Saskatchewan, Saskatoon, SK, S7N 5E6, Canada. {\tt fvk@math.usask.ca}

\smallskip \noindent 
Henri Lombardi. Laboratoire de Math\'ematiques, UMR CNRS 6623.
Universit\'e de Franche-Comt\'e, France.
{\tt henri.lombardi@univ-fcomte.fr}

\smallskip \noindent 
Herv\'e Perdry. Laboratoire de Math\'ematiques, UMR CNRS 6623.
Universit\'e de Franche-Comt\'e, France.
{\tt perdry@math.univ-fcomte.fr}

\newpage
\section*{Introduction} \label{sec Introduction}

\noi
We consider a valued field $\K$ with $\V$ its valuation ring and $\S$ a
subring of $\V$ such that $\K$ is the quotient field of $\S$. We assume
that $\S$ is an explicit ring and that divisibility inside $\V$ can be
tested, for any two elements of $\S$. By explicit ring we mean
a ring where algebraic operations and equality test are explicit. These
are our minimal assumptions of computability. If we want more
assumptions in certain cases we shall explicitly state them.

We let $\Kac$ denote the algebraic closure of $\K$ with $\Vac$ a
valuation ring that extends $\V$. Our general purpose is the discussion
of computational problems in $(\Kac,\Vac)$ under our computability
assumptions on $(\K,\V)$.

Each computational problem we shall consider has as input {\it a finite
family $(c_i)_{i=1,\ldots,n}$ of  parameters} in the ring $\S$. We call
them the {\it coefficients of our computational problem}.
Algorithms with the above minimal computability assumptions work
uniformly.
This means that some computations are made that give polynomials of
$\Z[C_1,\ldots,C_n]$,
and that all our tests are of the two following types:
$${\rm Is}\ P(c_1,\ldots,c_n) = 0\ ?   \qquad {\rm  Does}\
Q(c_1,\ldots,c_n)\ {\rm  divide}\ P(c_1,\ldots,c_n)\ {\rm in}\ \V \ ?
  $$
We are not interested in the way the answers to these tests are made.
We may imagine these answers given either by some oracles or by some
algorithms.

\markboth{Introduction}{Introduction}

\medskip
We shall denote the unit group by
$\,\UV \,$ or $\V^{\times}$, $\,\MV ={\V} \setminus \UV\, $  will be the
maximal ideal and $\,\UuV=1+\MV\,$ is the group of units whose residue
is equal to 1. We denote the value group $\K^{{\times}}/\UV$ by $\GK$.
We consider $\GKac$ as the divisible hull $\GKd$ of $\GK$, and the
valuation $v_\Kac$ as an extension of $v_\K$. We shall denote the
residue field $\V/\MV$ of $(\K ,{\V} )$ by $\,\ResK\,$. By convention,
$v(0)=\infty$ (this is not an element of $\GK$).

We say that the value of some element $x$ belonging to $\Kac$ is well
determined if we know an integer $m$ and two elements $F$ and $G$ of
$\Z[C_1,\ldots,C_n]$ such that, setting $f=F(c_1,\ldots,c_n)$, with
$f\neq 0$, and $g=G(c_1,\ldots,c_n)$, there exists a unit $u$ in
$\Vac$ such that:

$$f x^m = u g
$$
(a particular case is given by infinite value, i.e., when $x =0$.)

We call $v(x)$ the value of $x$ and we read the previous formula as:
$$ m\, v(x) = v(g)-v(f)
\>.$$
We shall use the notation $x\preceq y$ for $v(x)\leq v(y)$.
\begin{example} \label{example val group}
{\rm Let us for example explain the computations that are necessary to
compare $3v(x_1)+2v(x_2)$ to $7v(x_3)$ when the values are given by
$$f_1 x_1^{m_1} = u_1 g_1,\; f_2 x_2^{m_2} = u_2 g_2,\; f_3 x_3^{m_3} =
u_3 g_3,\quad (g_1,g_2,g_3\neq0)\>.
$$
We consider the LCM $m=m_1n_1=m_2n_2=m_3n_3$ of $m_1,m_2,m_3$.
We have that
$$f_1^{n_1} x_1^{m} = u_1^{n_1} g_1^{n_1},\; f_2^{n_2} x_2^{m} = u_2^{n_2}
g_2^{n_2},\; f_3^{n_3} x_3^{m} = u_3^{n_3} g_3^{n_3}\>.
$$
So  $3v(x_1)+2v(x_2)\leq 7v(x_3)$ iff
$g_1^{3n_1}g_2^{2n_2}f_3^{7n_3}\preceq f_1^{3n_1}f_2^{2n_2}g_3^{7n_3}$.
}
\end{example}

The reader can easily verify that computations we shall run in the value
group are always meaningful under our computability asumptions on the ring
$\S$.

In the same way, elements of the residue field will be in general
defined from elements of $\V$. So computations inside the residue
field are given by computations inside $\S$.

\medskip
The constructive meaning of the existence of an algebraic closure
$(\Kac,\Vac)$ of $(\K,\V)$ is that computations inside $(\Kac,\Vac)$
never produce contradictions. The constructive proof of this
constructive meaning
can be obtained  by considering classical proofs
(of the existence of an algebraic closure)
from the viewpoint of dynamical theories (see \cite{CLR}).

\medskip
The present paper can be read from a classical point of view as well as
from a constructive one. Our results give a uniform way for computing
inside $(\Kac,\Vac)$   when we know how to compute inside $(\K,\V)$.

\medskip In the first section we give some basic material for
computation inside algebraically closed valued fields. The most
important is the Newton Polygon Algorithm.

In  section  \ref{sec alg clos}, we explain how
the Newton Polygon Algorithm can be
used in order to make explicit computations inside the
algebraic closure of a valued field, even in the case where there is no
factorization algorithm for one variable polynomials. It is
sufficient to take the point of view of dynamic evaluations as in
\cite{DD}.

To conclude the paper, we give in section \ref{QE} a new quantifier
elimination algorithm for the theory of algebraically closed valued
fields (with fixed characteristic and residue field
characteristic). The geometric idea for this algorithm is simple.  It
can be easily implemented after the work done in section \ref{sec alg
clos}.

\section{Basic material} \label{sec Basic material}

\subsection{Multisets} \label{subsec multisets}

A {\it multiset } is a set with (nonnegative) multiplicities, or
equivalently, a list defined up to permutation. In particular, the roots
of a polynomial $P(X)$
form a multiset in the algebraic closure of the base field. We shall use
the notation $[x_1,\ldots,x_d]$ for the multiset corresponding to the list
$(x_1,\ldots,x_d)$. The {\it cardinality} of a multiset is the length of a
corresponding list, i.e., the sum of multiplicities occurring in the
multiset.

We shall use the natural (associative commutative) additive notation for
``disjoint unions" of multisets, e.g.,
$$[b,a,c,b,b,a,b,d,a,c,b]=3[a,b]+[b,b,d]+2[c]=3[a]
+5[b]+2[c]+[d]
\>.$$

We call {\it  a pairing between two multisets} what remains of a bijection
between  two corresponding lists when one forgets the ordering of the
lists. E.g., if we consider the two lists
$$(a,a,a,a',a',a',a'')=(a_i)_{i=1,\ldots,7}\quad\mbox{ and }\quad
(b,b,b',b',b'',b'',b'')=(b_i)_{i=1,\ldots,7}$$
corresponding to the multisets
$$3[a]+3[a']+[a'']\quad\mbox{ and }\quad 2[b]+2[b']+3[b'']\,,$$
and the bijection
$$a_1\mapsto b_3,\,a_2\mapsto b_4,\,a_3\mapsto b_1,\,
a_4\mapsto b_6,\,a_5\mapsto b_5,\,a_6\mapsto b_7,\,a_7\mapsto b_2\>,$$
then what remains can be described as
$$2[a\mapsto b']+[a\mapsto b]+3[a'\mapsto b'']+
[a''\mapsto b]\>,$$
or equivalently as
$$2[(a, b')]+[(a, b)]+3[(a', b'')]+[(a'', b)]\>.$$
This is a multiset of pairs that gives by the canonical projections
the initial multisets $3[a]+3[a']+[a'']$ and $2[b]+2[b']+3[b'']$.

This notion can be extended to $r$ multisets $M_1,\dots,M_r$ with same
cardinality $k$: a pairing between the $M_i$'s is a multiset of
$r$-tuples that gives by the canonical projections the initial
multisets $M_1,\dots,M_r$.

The notion of multisets is a natural one when dealing with roots of a
polynomial in an abstract setting. Multiplicity is relevant, but in
general there is no canonical ordering of the roots.

Dynamic evaluation in \cite{DD,DD2} can be understood as a way of
computing with root multisets.
\subsection{The Newton Polygon} \label{subsec NP}

Here we recall the well known \NPAz.

\smallskip The \pgn of a polynomial $P(X)=\sum_{i=0,\ldots,d} p_iX^i\in
\K[X]$ (where $p_d\not= 0$) is obtained from the list of pairs in
%
%
$\N\times (\GK\cup \left\{\infty\right\})$
$$
((0,v(p_0)),(1,v(p_1)),\ldots,(d,v(p_d)))\>.
$$
The \pgn is ``the bottom convex hull" of this list.
It can be formally defined as the extracted list
$\,((0,v(p_0)),\ldots,(d,v(p_d)))\,$ verifying:
two pairs $(i,v(p_i))$ and $(j,v(p_j))$ are two consecutive vertices
of the \pgn iff:

\centerline {if  $0\leq k< i\ $ then
$\ (v(p_j)-v(p_i))/(j-i)\ >\ (v(p_i)-v(p_k)/(i-k))$}

\centerline {if $ i<k< j\ $ then
$\ (v(p_k)-v(p_i))/(k-i)\ \ge\ (v(p_j)-v(p_i))/(j-i)$}

\centerline {if $j<k\le d\ $ then
$\ (v(p_k)-v(p_j))/(k-j)\ >\ (v(p_j)-v(p_i))/(j-i)$}

\noi
Let $P(X)=p_d\prod_{i=1}^d(X-x_i)$ in $\Kac[X]$.
It is easily shown that if  $(i,v(p_i))$ and $(j,v(p_j))$ are two
consecutive vertices in the \pgn of the polynomial $P$,
then the zeros of $P$
in $\Kac$ whose value in   $\GKd$ equals
$\ (v(p_i)-v(p_j))/(j-i)\ $ form a multiset with cardinality $\ j-i\ $.

\begin{proof}{Proof}
Order the $x_i$'s in
non-decreasing order of the values $v(x_i)$. We give the proof for an
example. Assume for instance that
$$ \nu_1=v(x_1)=v(x_2)<\nu_3=v(x_3)=v(x_4)=v(x_5)<\nu_6=v(x_6)\cdots
$$
Let us express $p_{d-j}/p_d$ as a symmetric function of the roots.
We see immediately that
$$\begin{array}{rcl}
v(p_{d-1})&\geq&v(p_d)+ \nu_1  \\
v(p_{d-2})&=   &v(p_d)+ 2\nu_1  \\
v(p_{d-3})&\geq&v(p_d)+ 2\nu_1+\nu_3>v(p_d)+ 3\nu_1 \\
v(p_{d-4})&\geq&v(p_d)+ 2\nu_1+2\nu_3  \\
v(p_{d-5})&=   &v(p_d)+ 2\nu_1+3\nu_3  \\
v(p_{d-6})&\geq&v(p_d)+ 2\nu_1+3\nu_3+\nu_6>v(p_d)+  2\nu_1+4\nu_3
\end{array}$$
So the two last edges of the \pgn are $((d-2,v(p_{d-2})),(d,v(p_d)))$
with slope $-2\nu_1$ and  $((d-5,v(p_{d-5})),(d-2,v(p_{d-2})))$
with slope $-3\nu_3$, giving the wanted result.
\end{proof}
%
Now we can give an answer to the following problem.

\begin{comput} \label{comput Val of roots}
{\em  (Multiset of values of roots of polynomials)}  \\
\Inp A polynomial $P\in \K [X]$ over a valued field $(\K ,{\V})$.
\\
\Output The multiset $[v(x_1),\ldots,v(x_n)]$ where $[x_1,\ldots,x_n]$ is
the multiset of roots of $P$ in $\Kac$.
\end{comput}
%
This problem is solved by the following algorithm, which is widely used
in the sequel.

\begin{proof}{\NPAz}
The number $n_\infty$ of roots equal to 0 (i.e., with infinite value)
is read off from $P$. Let $P_0:=P/X^{n_\infty}$. Compute the Newton
polygon of $P_0$, compute the slopes of the edges and output the answer.
\end{proof}

\subsection{Generalized Tschirnhaus transformation}
\label{subsec Tschirnhausen}
\def \A {{\bf A}}
\def \B {{\bf B}}

We recall a well known way of computing in algebraic extensions,
which we will use freely in our paper. We call this method the {\em
generalized Tschirnhaus transformation}.

Let $\K$ be a field, $(P_j)_{j=1,\ldots,r}$
be a family of monic
polynomials in $\K [X]$, and
$$P_j(X)=(X-\xi_{j,1})\cdot\ldots\cdot (X-\xi_{j,d_j})
$$
their factorizations in $\Kac[X]$. Take $Q(X_1,\ldots,X_r)\in
\K[X_1,\ldots,X_r]$, and let $d=d_1\cdots d_r$. We claim that the
polynomial
$$
T_Q(Z)=(Z-Q(\xi_{1,1},\ldots,\xi_{r,1}))\cdot\ldots\cdot
(Z-Q(\xi_{1,d_1},\ldots,\xi_{r,d_r}))
$$
of degree $d$ is the characteristic polynomial of $A_Q\,$, where $A_Q$
is the matrix of the multiplication by $Q(x_1,\ldots,x_r)$ inside the
$d$-dimensional $\K$-algebra
$$\K[x]:=\K [X_1,\ldots,X_r]/\gen{P_1(X_1),\ldots,P_r(X_r)}\;.$$

We give a proof of this well known fact, for which we found no
reference. We prove a slightly more general result, which deals with
roots of so-called {\em triangular systems}. Moreover, the computation
works in arbitrary commutative rings.

\begin{definition} \label{TRS}
    Let $\A\subset \B$ be commutative rings.
\begin{enumerate}
\item
   Take a system of polynomials
$$\begin{array}{c}
\P = (P_1,\ldots ,P_r)   \quad {\it  where} \;\;P_1(X_1)\in \A [X_1],\\
P_2(X_1,X_2)\in \A
[X_1,X_2],\,\ldots,P_r(X_1,\ldots,X_r)\in \A [X_1,\ldots,X_r]\>.
\end{array}$$
This system is called a {\em  triangular system}
if each $P_i$ is monic w.r.t.\ $X_i$.
\item
The {\em quotient algebra} is
$\A[X_1,\ldots,X_r]/\gen{P_1,\ldots ,P_r}=\A[x_1,\ldots,x_r]$ where
$x_i$ is the class of $X_i$. We denote it by  $\A_{\P}$.
Let $d_{i}=\deg_{X_{i}}(P_i)$.
Then  $\A_{\P}$ is a free $\A$-module of rank $d_1d_2\cdots d_r$ with
``monomial basis" $(x_1^{\mu_1}\cdots x_r^{\mu_r})_{\mu_i<d_i}$.
Note that we may assume w.l.o.g.\ that $\deg_{X_{j}}(P_k)<d_j$ for
$k>j$.
\item
A vector $\alpha=(\alpha_1,\ldots,\alpha_k)\in \B^{r}$ is called a
{\em root vector of} $\P$ (or {\em a solution of} $\P$) if
$$P_1(\alpha_1)=P_2(\alpha_1,\alpha_2)=\ldots=P_k(\alpha_1,\ldots,
\alpha_r)=0 \>.$$
\item
Assume for simplicity that $r=3$. We say that the system $\P$ fully
splits in $\B$ if $\B$ contains elements
$\xi_i$ $(i\leq d_1)$, $\xi_{i,j}$ $(i\leq d_1,\,j\leq d_2)$, and
$\xi_{i,j,k}$ $(i\leq d_1,\,j\leq d_2,\,k\leq d_3)$ such that
\begin{equation} \label{eqSplit}
\left.
\begin{array}{rclcl}
P_1(X)      & = &  \prod_{i\leq d_1} (X-\xi_i)   \\
P_2(\xi_i,Y)&= &   \prod_{j\leq d_2} (Y-\xi_{i,j})&& (i\leq d_1)   \\
P_3(\xi_i,\xi_{i,j},Z)
    &= &\prod_{k\leq d_3} (Z-\xi_{i,j,k})&& (i\leq d_1,\,j\leq d_2)
\end{array}
\right\}
\end{equation}
\item
When $\A=\K$ and $\B=\Kac$, two systems with the same variables are called
{\em  coprime systems} if
they have no common root vector.
\end{enumerate}
\end{definition}

In order to simplify notations, we give our result for the case $r=3$.
\begin{proposition}
\label{propTschir}
Let $\A\subset \B$ be commutative rings and
$\P = (P_1,P_2 ,P_3)$ a triangular system over $\A$ which fully splits
in $\B$ with equations $(\ref{eqSplit})$. Let $Q(x_1,x_2,x_3)\in\A_\P$,
$\mu_Q$ be the $\A$-linear endomorphism of $A_\P$ representing
multiplication by $Q$, and $C_Q(Z)$ the characteristic polynomial of
$\mu_Q$. Then we have
\begin{equation} \label{eqTschir}
C_Q(Z)=\prod_{i\leq d_1,\,j\leq d_2,\,k\leq d_3}
(Z-Q(\xi_{i},\xi_{i,j},\xi_{i,j,k}))
\end{equation}
\end{proposition}
\begin{proof}{Proof}
Note that we could have chosen $Q\in \A[X_1,X_2 ,X_3]$. But if
$(\alpha,\,\beta,\,\gamma)$ is a root vector of $\P$ in an extension of
$\A$,  it is clear that $Q(\alpha,\beta,\gamma)$ depends only of the
class of $Q$ in $\A_\P$, so equation (\ref{eqTschir}) is meaningful.

By Cayley-Hamilton $C_{Q}(\mu_Q)=0$ and since
$\mu_Q(1)=Q$, $C_{Q}(Q)=0$. This implies that
$C_{Q}(Q(\alpha,\beta,\gamma))=0$ each time we have a root vector
$(\alpha,\,\beta,\,\gamma)$ of $\P$ in an extension of $\A$ since
$\A[\alpha,\beta,\gamma]$ is a homomorphic image of $\A_\P$.

So the proposition is proved in the ``good case" where $\B$ is a
domain and all the root vectors in (\ref{eqSplit}) give distinct values
for $Q(\xi_{i},\xi_{i,j},\xi_{i,j,k})$: the RHS and LHS in
(\ref{eqTschir}) are monic univariate polynomials with the same roots,
all being distinct.

Now we give the proof for the ``generic case" where the $\xi_{i},
\xi_{i,j}, \xi_{i,j,k}$ and the coefficients $q_{i,j,k}$ of $Q$ are
indeterminates. This means that $\B$ can be replaced by a ring generated
over $\Z$ by these indeterminates, and $\A$ can be replaced by the
subring of $\B$ generated by the coefficients of $Q$ and by the
coefficients of $P_1,P_2,P_3$ which are defined by equations
(\ref{eqSplit}). In this generic case, $\B$ is an integral domain and
all the $Q(\xi_{i},\xi_{i,j},\xi_{i,j,k})$ are distinct. So the generic
case is a good case and we are done.

Finally, note that all non-generic cases are homomorphic
images of the generic case.
\end{proof}

We give another slight generalization, which can be proved in
a similar way.
Let $Q,R\in \A[X_1,\ldots,X_r]$ with $R(\xi)$ invertible in $\B$
for all the root vectors in (\ref{eqSplit}). Let $F=Q/R$.
Then $A_R$ is an invertible matrix (over $\B$) and the polynomial
$$
T_F(Z)=\prod_{i\leq d_1,\,j\leq d_2,\,k\leq d_3}
(Z-F(\xi_{i},\xi_{i,j},\xi_{i,j,k}))
$$
is the characteristic polynomial of $A_Q(A_R)^{-1}$.

\section{Dynamic computations in the algebraic closure}
\label{sec alg clos}
\noi Dynamic computations in the algebraic closure of a valued field are an
extension of dynamic computations in the algebraic closure of a field as
explained in \cite{DD,DD2}. First let us recall these ones.
\subsection{Dynamic algebraic closure} \label{subsec DAC}

The following algorithms tell us how to compute dynamically in the
algebraic closure of $\K $ when we do not want to (or we cannot) use
factorization algorithms in $\K [X]$.

First we examine the problem of adding one root of a monic polynomial
without factorization algorithm. If we are able to compute in the
field so
created, then we are able to compute recursively in any finite extension
given by adding one after the other roots of several polynomials.
In fact, since there is a priori an ambiguity about what root we have
introduced (distinct roots give in general non-isomorphic fields), we
have to compute all possible cases.

\begin{comput} \label{comput D5}
{\em (computational problem \`a la D5) } \

\noindent
\Inp Let $P$  (of degree $\geq 2$) and $Q$ be polynomials in $\K [X]$.\\
\Output Give correct answers to the following questions:
\begin{itemize}
\item [$(1)$] Is $Q$ zero at each root of $P$ in $\Kac$?
\item [$(2)$] Is  $Q$ nonzero at each root of $P$ in $\Kac$?
\item [$(3)$] If the two answers are ``No", compute two factors $P_1$ and
$P_2$ of $P$ and two polynomials $U_1$, $U_2$ such that: \\
--- $Q$ is zero at each root of $P_1$ in $\Kac$, \\
--- $Q$ is nonzero at each root of $P_2$ in $\Kac$, \\
--- $P_1$ and $P_2$ are coprime, $P_1U_1+P_2U_2=1$,\\
--- each root of $P$ in $\Kac$ is a root of $P_1P_2$.
\end{itemize}
\end{comput}
We give two natural solutions of the previous problem.
\begin{proof}{Algorithm SquarefreeD5}
{\em  (solving computational problem \ref{comput D5} when $P$ is a
squa\-re\-free polynomial)}\\
Assume that $P$ is squarefree.

\noi  Compute the monic GCD $P_1$ of $P$ and $Q$.

\noi If $P_1=1$ then answer ``Yes'' to the second question;

\noi $\hspace{5.5mm}$ else if ${\rm lc}(P) P_1=P$ then answer ``Yes'' to
the first question;

\noi $\hspace{11mm}$ else return $P_1$, $P_2:= P/P_1$ and polynomials
$U_1$, $U_2$ s.t.\ $P_1U_1+P_2U_2=1$.
\end{proof}

\begin{proof}{Algorithm BasicD5}
{\em (solving computational problem \ref{comput D5})} \\
Compute the monic GCD $P_1$ of $P$ and $Q$.

\noi If $P_1=1$ then answer ``Yes'' to the second question;

\noi $\hspace{5.5mm}$ else compute the monic polynomial $P_2$ such that:

\noi $\hspace{5.5mm}$ $P_2$ divides  $P$, ${\rm GCD}(P_1,P_2)=1$ and $P$
divides  $P_1^mP_2$ (for some $m$);

\noi $\hspace{11mm}$ if $P_2=1$ then answer ``Yes'' to the first
question, and replace $P$ by $P_1$;

\noi $\hspace{16.5mm}$ else return $P_1$, $P_2$ and polynomials $U_1$, $U_2$
s.t.\ $P_1U_1+P_2U_2=1$.
\end{proof}
The replacement of $P$ by $P_1$ is not used in the algorithm itself, but
is meant for use by subsequent algorithms because if $P_2=1$ then $P_1$
has the same roots as $P$ but possibly smaller degree.

\begin{remark}
\label{remBasicD5}
{\rm
Observe
that $P_2=P/\GCD(P_1^k,P)=P/\GCD(Q^k,P)$ where $k=1+\deg(P)-\deg(P_1)$.
We can also get $P_2$ by iteration of the process: start with $R=P$;
replace $R$ by $R/\GCD(R,Q)$ (here $\GCD(R,Q)$ means the monic GCD of
$R$ and $Q$), until the GCD is $1$.

If $P$ is monic and the ring $\S$ is normal then
$P_1$ and $P_2$ are in $\S[X]$, but it is not always easy to make this
result explicit. Nevertheless we can always compute $P_1$ and $P_2$ using
coefficients in the quotient field of $\S$: the GCD computation may use
pseudo divisions instead of divisions. The use of subresultant polynomials
may improve the efficacity of the algorithm.
}
\end{remark}

We can understand the previous algorithms as breaking the set of roots
of a polynomial in distinct subsets anytime that some objective
distinction may be done between the roots. Their stupendous simplicity
is certainly the main reason explaining their non-universal use in the
literature about algebraic extensions of fields.

\begin{remark}
\label{remRootmultiset}
{\rm If we see the roots of $P$ as a multiset, and if we want to
keep the information concerning multiplicities, the output
\begin{itemize}
\item $(P_1,P_2)$ with $P_1$, $P_2$ coprime and each root of $P$
in $\Kac$ is a root of $P_1P_2$.
\end{itemize}

\noi is not the good one. We need in this case one of the two
following outputs:
\begin{itemize}
\item $(P_1,P_2)$ with $P_1$, $P_2$ coprime and  $P_1P_2=P$.
\end{itemize}

\noi  or in a more economic way for future computations:
\begin{itemize}
\item $(P_1,P_2)$ with $P_1$, $P_2$ coprime,  $P_1P_2=P$ and a
decomposition of each $P_i$ as a product
of powers of coprime polynomials.
\end{itemize}

\noi
The computational problem corresponding to the first output can be
solved by the following slight variant of {\bf  BasicD5}.

\begin{proof}{Algorithm MultisetD5} \emph{(solving a multiset variant of
computational problem~\ref{comput D5}).}

\noi \Inp Let $P$  (of degree $\geq 2$) and $Q$ be polynomials in $\K [X]$.\\
\Output $(P_1,P_2)$ with $P_1$, $P_2$ coprime,  $P_1P_2=P$, $Q$ is zero at
each root of $P_1$ in $\Kac$, $Q$ is nonzero at each root of $P_2$ in
$\Kac$.

\noi  Compute the monic GCD $R_1$ of $P$ and $Q$.

\noi If $R_1=1$ then return $P_1=1$, $P_2=P$

\noi $\hspace{5.5mm}$ else compute the monic polynomial $P_2$ such that:

\noi $\hspace{11mm}$ $P_2$ divides  $P$, ${\rm GCD}(R_1,P_2)=1$ and $P$
divides  $R_1^mP_2$ (for some $m$).

\noi $\hspace{5.5mm}$ return $P_2$, $P_1=P/P_2$ and polynomials $U_1$, $U_2$
such that $P_1U_1+P_2U_2=1$.
\end{proof}

}
\end{remark}

We now explain the recursive use of algorithms {\bf  SquarefreeD5} and
{\bf BasicD5}. Note that root vectors of a triangular system $\P$ as in
definition~\ref{TRS} form a multiset of cardinality
$d=\prod_i\deg_{X_i}(P_i)$.

\begin{comput} \label{comput recursive D5}
{\em  (computing in extensions generated by several successive algebraic
elements)} \\
\Inp
\begin{itemize}
\item  A triangular system of polynomials
$\P = (P_1,\ldots ,P_n)$:
$$P_1(X_1)\in \K [X_1],\,P_2(X_1,X_2)\in \K
[X_1,X_2],\,\ldots,P_k(X_1,\ldots,X_k)\in \K [X_1,\ldots,X_k]\>.
$$
\item A finite list of polynomials $Q_1,\ldots,Q_r$ in $\K
[X_1,\ldots,X_k]$.
\end{itemize}
\Output
\begin{itemize}
\item  A list of coprime triangular systems
$\overline{S^{(1)}},\ldots,\overline{S^{(\ell)}}$
whose root vectors form a partition of the set of all solutions of the
initial triangular system $\P$, such that for each $j$, the $r$-tuple of
signs for the tuple $(Q_1(\x),\ldots,Q_r(\x))$ (the sign of $y$ is
either $0$ if $y=0$ or $1$ if $y\neq 0$), is the same for every root
vector $\x=(x_1,\ldots,x_k)$ of $\overline{S^{(j)}}\,$.
\item  For each  triangular system $\overline{S^{(j)}}$, this
fixed $r$-tuple of signs.
\end{itemize}
\end{comput}
In the general case, we can solve the previous problem in the following
way.
\begin{proof}{Algorithm TriangularBasicD5}
{\em (solving computational problem \ref{comput recursive D5})}

\noindent Use {\bf BasicD5} recursively. More precisely, consider
that $Q$ and $P_k$ are polynomials in the variable $X_k$ with parameters
$(x_1,\ldots,x_{k-1})$. When making the computations of {\bf BasicD5}
we have to solve some tests

\centerline{ ``\ Is $R(x_1,\ldots,x_{k-1})$ equal to zero or not\ ?\ " }

\noindent for some polynomials $R$ given by the computation. So we have
to solve the same kind of problem with one variable less. Hence, a
recursive computation will produce the answer.
\end{proof}

In the case of a perfect field, we can use {\bf SquarefreeD5}
recursively. To see why this works, we have to recall how to compute the
squarefree part of a polynomial in one variable in this case.
\begin{proof}{Algorithm SquarefreePart}
{\em (compute the squarefree part of a polynomial in one variable in the
case of a perfect field)}

\noi We assume that $\K$ is a perfect field. In the characteristic $p$
case we assume that getting $p$-th roots is explicit inside $\S$.

\sni \Inp A polynomial $P \in\S[X]\,$.

\noi \Output $P_1$ the squarefree part of $P$.

\sni If the characteristic is zero then $P_1=P/\GCD(P,P')$.

\sni If the characteristic is $p$ then let  $P_1=1$ and:

\noi Iterate the following process:

\noi $\hspace{5.5mm}$ Beginning with $R=P$ iterate the following process:

\noi $\hspace{11mm}$        If $R=Q(X^p)$ then replace $R$ by $R^{1/p}$
else replace $R$ by $R/\GCD(R,R')$

\noi $\hspace{5.5mm}$ until you find $\GCD(R,R')=1$.

\noi $\hspace{5.5mm}$ Replace $P_1$ by $P_1\cdot R$

\noi $\hspace{5.5mm}$ Iterate the following process:

\noi $\hspace{11mm}$        Replace $P$ by $P/\GCD(P,R)$

\noi $\hspace{5.5mm}$ until  you find $\GCD(P,R)=1$

\noi until $P=1$.
\end{proof}
We suggest that the reader apply the algorithm to a polynomial of the
form $Q_1(X^p)^2Q_2(X^{p^2})$ with $p\ne 2$, in order to see why the
loops in this algorithm are necessary.

\begin{proof}{AlgorithmPerfectTriangularD5}
{\em (solving computational problem \ref{comput recursive D5} in the case
of a perfect field)}

\noindent We assume that $\K$ is a perfect field. In the characteristic
$p$  case we assume that getting $p$-th roots is explicit inside $\S$.

\sni In a first big step we replace the initial system by a disjunction of
coprime systems that are ``squarefree".

\noi For each polynomial in the triangular system,
we use {\bf  SquarefreePart}  and (recursively)
{\bf  SquarefreeD5} to replace it by a ``squarefree" polynomial.

\noi More precisely, first we replace $P_1$ by its squarefree part $S_1$.

\noi Then we try to apply {\bf  SquarefreePart} to the polynomial $P_2$
as if the quotient algebra $\K [X_1]/S_1(X_1)$ were a field.
If this is not possible, {\bf  SquarefreeD5} produces a splitting of
$S_1$. In each branch so created the computation is possible and we can
replace $P_2$ by its squarefree part.

\noi For example, we may get three branches with the following
properties. In the first one, the squarefree polynomial $P_{1,1}$
replaces $P_1\,$, and $P_2$ is already squarefree, so that
$P_{2,1}=P_2\,$. In the second one, the squarefree polynomial $P_{1,2}$
replaces $P_1\,$, and the squarefree part of $P_2$ is given by $P_{2,2}$
with degree ${\rm deg}(P_2)-1 $. In the third one, $P_{1,3}$ replaces
$P_1$ and the squarefree part of $P_2$ is given by $P_{2,3}$ with degree
${\rm deg}(P_2)-4 $. Then we introduce $P_3$ in every branch previously
created and try to apply {\bf SquarefreePart} to the polynomial $P_3$ as
if the corresponding quotient algebra $\K [X_1,X_2]/\gen{P_{1,i}(X_1),
P_{2,i}(X_1,X_2)}$ were a field. If this is not possible,
{\bf SquarefreeD5} produces a splitting of $P_{1,i}$ or $P_{2,i}$.

\noi And so on.

\noi When we have introduced all $P_i$'s, we get a tree. Each leaf of
the tree corresponds to a new triangular system where all successive
polynomials replacing the $P_i$'s are ``strongly squarefree" (the
squarefreeness is certified by a Bezout identity in the suitable quotient
algebra). Distinct leaves correspond to coprime triangular systems.
So the set of root vectors of $\P$ is partitioned into distinct subsets,
each one corresponding to a leaf of the tree.

\sni Now we describe the second ``big step".
At each leaf of the tree we search for the signs of the $Q_j$'s
using {\bf  SquarefreeD5} as if the corresponding quotient algebra were a
field. If this is not possible, new splittings are produced.
\end{proof}

\begin{remark}
\label{remMultisets}
{\rm Slight variants of the above algorithms give a partition of the
{\em multiset} of solutions of the triangular system $\P$ in disjoint
multisets that are defined by coprime triangular systems
$\overline{{S'}{}^{(j)}}$, each $Q_i$ having a constant sign at
the zeros of each $\overline{{S'}{}^{(j)}}$.
}
\end{remark}

\begin{remark} \label{rem QEACF}
{\rm  The above algorithms can be generalized in order
to search systematically for solutions of any
system of sign conditions: equalities need not be in a triangular form.
So they can be seen as quantifier elimination algorithms
in the first order theory of algebraically closed extensions
of some explicitly given field $\K $.
}
\end{remark}

In the following subsection we show that the same kind of computations are
possible in the case of valued fields.

\subsection{Dynamic algebraic closure of a valued field} \label{subsec suc
ref}

\subsubsection*{Roots of one polynomial}

The valued algebraic closure of $(\K ,{\V})$ is well determined up to
isomorphism. So the following computational problem makes sense.
\begin{comput} \label{comput sim val} {\em  (Simultaneous values)} \\
\Inp polynomials $P$ (monic) and $Q_1,\ldots,Q_r$ in $\K[X]$.
Call $[x_1,\ldots,x_d]$ the multiset of roots of $P$ in
$\Kac$.

\noi \Output The multiset $[\left(v(x_i),v(Q_1(x_i)),\ldots,v(Q_r(x_i))
\right) ]_{i=1,\ldots,d}$ of $(r+1)$-tuples of values.
\end{comput}
%
This problem is solved by the following algorithm.

\begin{proof}{Algorithm SimVal}
{\em (solving computational problem \ref{comput sim val})}

\noi We start with the case $r=1$. Assume w.l.o.g.\ that $P(0)\neq0$.
The multiset $[\nu_i]_{i=1,\ldots,d}$ of (finite) values of the $x_i$'s
is given by the \NPA for $P$.

\sni
For $m,n\in \N$, the polynomial
$$S_{m,n}(X)=(X-x_1^mQ_1(x_1)^{n})\cdot\ldots\cdot
(X-x_d^mQ_1(x_d)^{n})
$$
is the characteristic polynomial of the matrix $A^m(Q_1(A))^{n}$ where $A$
is the companion matrix of $P$.

\noi So, using the Newton polygon of $S_{m,n}$ we know the multiset
$$[m\,v(x_i)+n\,v(Q_1(x_i))]_{i=1,\ldots,d} =
[m\,\nu_i+n\,\nu_{1,i}]_{i=1,\ldots,d}
$$
for any $(m,n)$.

\noi We compute first the multiset $[\nu_{1,i}]_{i=1,\ldots,d}\>$.

\noi We want to compute the correct pairing between the two multisets
$[\nu_i]_{i=1,\ldots,d}$ and $[\nu_{1,i}]_{i=1,\ldots,d}\>$.

\noi Assume first that no $\nu_{1,i}$ is infinite.

\noi Let us call a {\it  bad coincidence for $n_1$} an equality
$$ \nu_i+n_1\,\nu_{1,h}=\nu_j+n_1\,\nu_{1,k} \quad {\rm with}\ \
\nu_i\not= \nu_j,
\quad i,j,h,k\in\left\{1,\ldots,d\right\}
\>.$$
If there is no bad coincidence for some $n_1$ then we can state this
fact by considering the two sets $\{\nu_i\ :i=1,\ldots,d\}$ and
$\{\nu_{1,i}\ :i=1,\ldots,d\}$. Note also that there are at most
$(d(d-1)/2)^2$ ``bad values" of $n_1$. So we can find a ``good" $n_1$
by a finite number of computations.
Fix a ``good" $n_1$. From the multisets $[\nu_i]_{i=1,\ldots,d}$ and
$[\nu_{1,i}]_{i=1,\ldots,d}$ we deduce the multiset
$[\nu_i+n_1\,\nu_{1,j}]_{i=1,\ldots,d,j=1,\ldots,d}$.
Now, $n_1$ being ``good'', the multiset
$[\nu_i+n_1\,\nu_{1,i}]_{i=1,\ldots,d}$ (obtained by the \NPA applied
to $S_{1,n_1}$) can be read as a submultiset of
$[\nu_i+n_1\,\nu_{1,j}]_{i=1,\ldots,d,j=1,\ldots,d}\,$.  This gives us
the pairing between the multisets $[\nu_i]_{i=1,\ldots,d}$ and
$[\nu_{1,i}]_{i=1,\ldots,d}\,$.\\
For example, assume that
$$[\nu_i]_{i=1,\ldots,9}=3[\alpha_1]+4[\alpha_2]+2[\alpha_3],\qquad
[\nu_{1,i}]_{i=1,\ldots,9}=2[\beta_1]+2[\beta_2]+2[\beta_3]+3[\beta_4]$$
and that the number 5 is good, i.e., the twelve values
$\alpha_i+5\beta_k$ are distinct. Computing the multiset
$[\nu_i+5\,\nu_{1,i}]_{i=1,\ldots,9}$, we find, e.g.,
$$\begin{array}{c}
[\alpha_1+5\beta_1]+ 2[\alpha_1+5\beta_4]+ [\alpha_2+5\beta_4]+
  2[\alpha_2+5\beta_2]+ \\
\;\;+[\alpha_2+5\beta_3] + [\alpha_3+5\beta_1]+ [\alpha_3+5\beta_3]
,\end{array}$$
and we get the pairing
$$
  [(\alpha_1,\beta_1)]+ 2[(\alpha_1,\beta_4)]+ [(\alpha_2,\beta_4)]+
  2[(\alpha_2,\beta_2)]
+[(\alpha_2,\beta_3)] + [(\alpha_3,\beta_1)]+ [(\alpha_3,\beta_3)]\>.
$$

\noi {\em Comment}: the multiset $[x_i]_{i=1,\ldots,d}$ is, as a root
multiset, made of ``indiscernible elements".
The knowledge of the multiset $[\nu_i]_{i=1,\ldots,d}$ introduces some
distinction between the roots (if the $\nu_i$'s are not all equal).
The knowledge of the multiset
$[\nu_i+n_1\,\nu_{1,i}]_{i=1,\ldots,d}$ (with a ``good" $n_1$) induces
a finer distinction between the roots.

\sni We remark that the case where some $Q_1(x_i)$'s equal zero can also
be done correctly by a slight modification of the previous algorithm.
Nevertheless, when such a case appears, it seems more natural to
use the technique of dynamical evaluation (see \cite{DD} and section
\ref{subsec DAC}). If not all $Q_1(x_i)$'s equal zero (which is a
trivial case), then one can compute a factorization of $P$ in a product
of two coprime polynomials $P_1$ and $P_2$
by applying algorithm {\bf BasicD5} to $P$ and $Q_1$. Then we can study
separately the roots of these two polynomials. Moreover, the following
steps of the algorithm are clearer if all $Q_1(x_i)$'s are distinct from
zero.

\sni Next we show that analogous arguments work for the general case.
It will be sufficient to show how the case $r=2$ works.
Set $\nu_{2,i}=v(Q_2(x_i))$.
We  have computed the correct pairing
$[(\nu_1,\nu_{1,1}),(\nu_2,\nu_{1,2}),\ldots ,(\nu_d,\nu_{1,d})]$
between the multisets
$[\nu_i]_{i=1,\ldots,d}$ and $[\nu_{1,i}]_{i=1,\ldots,d}$.
We know also a ``good" integer $n_1$.
We can assume w.l.o.g. that all $\nu_{1,i}$'s and $\nu_{2,i}$'s are
finite.
We compute first the multiset $[\nu_{2,i}]_{i=1,\ldots,d}$.
Let us call a {\it  bad coincidence for $n_2$} an equality
$$ \nu_i+n_1\,\nu_{1,i}+n_2\,\nu_{2,h}=\nu_j+n_1\,\nu_{1,j}+n_2\,\nu_{2,k}
\quad {\rm with}\ \  \nu_i+n_1\,\nu_{1,i}\not= \nu_j+n_1\,\nu_{1,j}
\>.$$
If there is no bad coincidence for some $n_2$ then we can state this
fact
by considering the two sets $\{\nu_i+n_1\,\nu_{1,i}\ :i=1,\ldots,d\}$ and
$\{\nu_{2,i}\ :i=1,\ldots,d\}$.
We choose such an integer $n_2$. And so on.
\end{proof}

\begin{remark} \label{rem Disc elem}
\rm\parindent0mm

Assume that $P$ is a squarefree polynomial, so the $x_i$'s are in the
separable closure $\K ^{\rm sep}$ of $(\K ,{\V})$.  Assume that
algorithm {\bf SimVal} has shown that some list of values
$(\nu_i,\nu_{1,i},\ldots,\nu_{r,i})$ corresponds to only one root of
$P$.  It is clear from the abstract definition of the henselization
that such a ``discernible" element over $(\K ,{\V})$ is inside the
henselization $\Kh$ of $(\K ,{\V})$.  A perhaps surprising
computational consequence is that, since the henselization is an
immediate extension, when algorithm {\bf SimVal} isolates (or
discerns) some root of $P$, then the corresponding list of values is
made only of ``integer values'', i.e., values of elements of $\K$
``without integer denominator''. We can prove this constructively:

First, using computations in the henselization $\Kh$ as defined in
\cite{KL}, one can prove (cf.~\cite{lapin}) the following lemma:

\begin{lemma}
If the polynomial $P\in \Kh[X]$ has roots $x_1,\dots,x_d$ and if the
$d$-tuple
$[v(Q(x_1)),\dots,v(Q(x_d))]$ (provided by {\bf SimVal} applied to
$P,Q$ or by any other way) is equal to $d_1[\alpha_1]+ \dots+
d_k[\alpha_k]$, with $\alpha_i\neq\alpha_j$ (for $i\neq j$),
then one can factorize $P=P_1\dots P_k$ in $\Kh[X]$ ($\deg P_i=d_i$),
such that, if the roots of $P_i$ are $y_1,\dots,y_{d_i}$ the
$d_i$-tuple $[v(Q(y_1),\dots,v(Q(y_{d_i}))]$ is equal to
$d_i[\alpha_i]$.
\end{lemma}

Then if some list of values $(\nu_i,\nu_{1,i},\dots,\nu_{r,i})$
corresponds to only one root of $P$, we let\\
$n_0=\#\{ j :\nu_j=\nu_i \}$,\\
$n_1=\#\{ j : \nu_j=\nu_i \mbox{ and }\nu_{1,j}=\nu_{1,i}\}$,\\
\dots\\
$n_r=\#\{ j : \nu_j=\nu_i \mbox{ and }\nu_{k,j}=\nu_{k,i}
\ \ k=1,\dots,r \}=1$\\

{\def\ab{\allowbreak}
The previous result applied to $P(X)$ and $Q(X)=X$ provides a factor
$P_0$ of $P$, with degree $n_0$; then applied to $P_0(X)$ and $Q_1(X)$,
it provides a factor $P_1$ with degree $n_1$, and so on. Finally, we
obtain a factor $P_r$ of degree $n_r=1$. So the corresponding root is in
$\Kh$. The computations in $\Kh$ prove that the list of
values is made only of ``integer values''; one can compute explicitly
elements of $\K$ having the same value. More precisely, one can
compute $z_0,z_1,\dots,z_r\in\K$, such that $x_i=z_0(1+\nu_0),\ab
Q_1(x_i)=z_1(1+\nu_1),\ab\dots,\ab Q_r(x_i)=z_r(1+\nu_r)$, with
$v(\nu_i)>0$ for all $i$.}
\end{remark}

\subsubsection*{Root vectors of triangular systems}

Algorithm {\bf SimVal} says that ``we can compute in $\K [x]$"
where $x$ is a root of $P$ satisfying certain ``compatible value
conditions".  We know how many roots of $P$ correspond to a system of
compatible value conditions. Computing in $\K [x]$ means that we can
get ``any brute information concerning the valuation in this field",
more precisely, we can decide, for any new polynomial $Q$, if the
value of $Q(x)$ is well determined or not. And we can compute the
value(s).  When several possibilities for $v(Q(x))$ appear, choosing
one possible value, we refine our description of $\K [x]$.

So even if $\K [x]$ is not a priori a completely well determined valued
field, we can nevertheless always do as if it was completely well
determined.
And we get recursively the following computations, exactly as in section
\ref{subsec DAC}.

More precisely, our computational problem is the following.

\begin{comput} \label{comput severalg} \ \\
{\em (computing in extensions generated by several successive
algebraic elements)}

\noi \Inp
\begin{itemize}
\item  A triangular system of polynomials
$\P = (P_1,\ldots ,P_n)$:
$$P_1(X_1)\in \K [X_1],\,P_2(X_1,X_2)\in \K
[X_1,X_2],\,\ldots,P_k(X_1,\ldots,X_k)\in \K [X_1,\ldots,X_k]
\>.$$
\item A finite list of polynomials $Q_1,\ldots,Q_r$ in $\K
[X_1,\ldots,X_k]$.
\end{itemize}

\noi \Output
\begin{itemize}
\item The multiset of $(k+r)$-tuples of values
$$[\left(v(x_1),\ldots,v(x_k),v(Q_1(\x)),
\ldots,v(Q_r(\x))\right)]_{\x=(x_1,\ldots,x_k)\in R}$$
where $R$ is the multiset of root vectors of $\P$
(this multiset has cardinality $d=\prod_i\deg_{X_i}(P_i)$).
\end{itemize}
\end{comput}
%
This problem is solved by the following algorithm.

\begin{proof}{Algorithm TriangularSimVal}
Use recursively algorithm {\bf SimVal}.
\end{proof}

\subsubsection*{Graph of roots}
The following algorithm can be seen as a particular case of the previous
one. We denote by $\mu(P,a)$ the multiplicity of $a$ as root of the
univariate polynomial $P$ (if $P(a)\neq 0$ we let $\mu(P,a)=0$).
\begin{comput} \label{comput graphroot} \ \\
{\em (computing the ultrametric graph of roots of a family of univariate
polynomials)}

\noi \Inp
\begin{itemize}
\item A finite family of univariate polynomials
$\P = (P_1,\ldots ,P_s)$ in $\K [X]$.
\end{itemize}

\noi \Output
\begin{itemize}
\item The number $N$ of distinct roots of $P_1\cdots P_n$.
\item For some ordering $(x_1,\ldots,x_N)$ of these roots the finite family
$$ \left( (\mu(P_i,x_j))_{i\in[1,s],j\in[1,N]}, (v(x_j-x_\ell))_{1\leq
j<\ell\leq N}\right)
\>.$$
\end{itemize}
\end{comput}
Note that there are many possible answers, by changing the order of the
roots. All correct answers are isomorphic.
\begin{proof}{Algorithm GraphRoots}
First a recursive use of {\bf BasicD5} allows to find a finite
multiset of pairwise coprime polynomials $(R_1,\ldots,R_r)$ such that
each $P_i$ is a product of some $R_k$'s.
%
So we can assume w.l.o.g. that the $P_i$'s are pairwise coprime. If
$\deg(P_i)=n_i$ we introduce the roots
$x_{i,1},\ldots,x_{i,n_{i}}$ of $P_i$ through the triangular system\\
\vskip\abovedisplayskip
\def\eqq{\hbox to 0pt{\hss$=$\hss}}
\noi\hbox to \hsize{\hss$\begin{array}{rcl}
P_{i,1}(X_{i,1})&\eqq&P_i(X_{i,1}) \\
&&\\
P_{i,2}(X_{i,1},X_{i,2})&\eqq&
\frac{ P_{i,1}(X_{i,2})-P_{i,1}(X_{i,1}) }
{X_{i,2}-X_{i,1}} \\
&&\\
P_{i,3}(X_{i,1},X_{i,2},X_{i,3})&\eqq&
\frac {P_{i,2}(X_{i,1},X_{i,3})-P_{i,2}(X_{i,1},X_{i,2})}
{X_{i,3}-X_{i,2}}\\
\vdots\qquad \qquad & \hbox to 0pt{\hss\vdots\hss}  & \qquad \qquad \vdots
\\
P_{i,n_{i}}(X_{i,1},\ldots,X_{i,n_{i}})&\eqq&
  \frac {P_{i,n_i-1}(X_{i,1},\ldots,X_{i,n_i-2},X_{i,n_i})-
P_{i,n_i-1}(X_{i,1},\ldots,X_{i,n_i-2},X_{i,n_i-1})}
{X_{i,n_i}-X_{i,n_i-1}}\\
&&\\
P_{i,1}(x_{i,1})&\eqq&0\\
P_{i,2}(x_{i,1},x_{i,2})&\eqq&0\\
P_{i,3}(x_{i,1},x_{i,2},x_{i,3})&\eqq&0\\
\vdots\qquad\qquad& \hbox to 0pt{\hss\vdots\hss}  &  \vdots \\
P_{i,n}(x_{i,1},\ldots,x_{i,n_{i}})&\eqq&0
\end{array}$\hss}\break
\vskip\belowdisplayskip

The $P_{i,k}$'s give all together  a triangular system and we can apply
{\bf TriangularSimVal} for finding the values $v(x_{i,k}-x_{i',k'})$.
We remark that we can use a simplified form of {\bf TriangularSimVal}
since all possible results are isomorphic and we need only one of these
results. E.g., in the first step we compute the multiset
$[(v(x_{1,k}-x_{1,k'})_{1\leq k<k'\leq n_1}]$ but we select arbitrarily
one value as the good one w.r.t.\ some ordering of the roots, and so on.
\end{proof}
\begin{remark}
\rm\parindent0mm
There are probably some shortcuts allowing to give this ultrametric
graph in a quicker way: for example, for a single polynomial, it is
easy to compute the multiset of values $[v(x_i-x_j)]_{i\neq j}$ {\it
without}\/ knowing exactly to which edge each value corresponds; there
might be a way (at least in a great number of cases) to reconstruct the
graph (up to isomorphism).
\end{remark}

%
%
\section{Quantifier elimination}
\label{QE}
%


The aim of this section is to give a transparent proof of the
following well known theorem (cf.\ \cite{Wei}).

\begin{theorem} \label{belette}
The theory of algebraically closed valued fileds (with fixed
characteristics) admits quantifier elimination.
\end{theorem}

First we give a sketch of the proof of this theorem.
Our algorithm is a kind of ``cylindric algebraic decomposition"
(in the real closed case see, e.g., \cite{BCR}).
Given a finite set of multivariate polynomials, we choose a variable
as being the main variable and we consider the other ones as
parameters.

We settle in subsection \ref{subsecUEDP} an existential
decision procedure for a quantifier free formula with only one
variable:
given a finite set $S$ of univariate polynomials, we give a complete
description of the ``valued line $\Kac$" w.r.t. $S$.

More precisely, we give first a formal name to
each root of each polynomial in $S$, and we compute the ultrametric distance
between each pair of these roots. We compute also the multiplicities
of these roots and all the values
$v(P_i(x_j))$ for each root $x_{j}$ and each polynomial $P_i$.
This job is done by algorithm {\bf GraphRoots}.

Next, from these datas, we are able to test if a given conjunction of
elementary assertions concerning the $v(P_i(\xi))$'s
is realizable by some $\xi$ of the line $\Kac$.
In order to make this test we need a key geometric lemma,
concerning {\it ultrametric graphs}.
We explain this lemma in section \ref{subsecUG}.

The structure of our existential univariate decision procedure is
very simple.
This implies a kind of uniformity in such a way that
the algorithm can be performed ``with parameters'', exactly
as {\bf BasicTriangularD5} is nothing but a parametrized version
of {\bf Basic D5}. This gives a good way for eliminating the quantifier
in a formula with only one existential quantifier. So the work done
in our final section \ref{subsecQEL} will be a careful verification of
uniformity for the algorithms used in section \ref{subsecUEDP}.

Finally, the general elimination procedure follows by usual tricks.

\ms We now give general explanations about notations and technical
tools needed in the algorithms.

As in \cite{Wei} we use a two-sorted language,
$L=(L_F,L_\Gamma,v)$. The language of fields $L_F=\{
0,1,+,-,.\}$ is the $F$-sort. The language $L_\Gamma$ is the
$\Gamma$-sort. There is one more symbol, $v$, which is a
function symbol for the valuation.  The language $L_\Gamma$
consists of the language $L_\Gamma'=\{0,\infty,+,-,< \}$ of ordered
Abelian groups with last element $\infty$ together with a family of
symbols $\{ {\cdot\over q} : q\in\mathbb{N}^*\}$.

By convention $a-\infty =0$ for all $a\in\Gamma$. But there are some
ambiguities as $a-(b-c)$ may not be equal to $a-b+c$. In fact, it is
possible to avoid the sign $-$ for $\Gamma$-formulas, using case
distinctions. For example, we can replace $a-b=c$ by $(b=\infty \land
c=0)\lor a=b+c$. So any quantifier free formula $\Phi$ is equivalent
to a formula written without the $\Gamma$-sign $-$.
In the sequel we assume w.l.o.g.\ that $\Gamma$-terms are always written
without using the $\Gamma$-sign $-$.

Note also that we have no function symbol for the inverse of a nonzero
element inside the field. This is not a restriction. The introduction of
this function symbol would imply some trouble as the necessity of some
strange convention as $x/0=0$ for any $x$.

The theory of algebraically closed non-trivial valued fields is
{\bf ACVF}$(L)$. Recall that the formal theory specifies the
characteristic of the field and of the residue field. In our formulas
there are $F$-variables and $\Gamma$-variables, $F$-terms and
$\Gamma$-terms, and, more important, $F$-quantifiers and
$\Gamma$-quantifiers.

The rules of building terms are the natural ones. We see that the
$F$-terms are formal polynomials in $\mathbb{Z}[x_1,\cdots,x_n]$.
For the $\Gamma$-terms, we avoid the $\Gamma$-sign $-$. Take
$r_1,\dots,r_k\in\mathbb{Q}^{>0}$, and let $f_1,\dots,f_\ell$
(with $\ell\leq k$) be $F$-terms; then
\begin{equation}\label{gammaterm}
%
%
r_1\cdot v(f_1) + \cdots + r_\ell \cdot v(f_\ell)
+ r_{\ell+1}\cdot a_{\ell+1}+\cdots+r_k\cdot a_k
\end{equation}
(where each $a_i$ is a $\Gamma$-variable or a $\Gamma$-constant)  is a
general $\Gamma$-term.
Moreover we remark that such a
$\Gamma$-term can be easily rewritten as
%
$${1\over N}\;( v(f)+ s_{\ell+1}\cdot a_{\ell+1}+
\cdots+s_k\cdot a_k)
$$
where $N,s_j\in \mathbb{Z}^{>0}$.

When we want to make computations inside the algebraic closure of some
explicitly given valued field $(\K,\V)$ we have to use the theory
{\bf ACVF}$(\K,\V)$ where the elements of $\K$ and $\GK$ are added as
constants and the diagram of the valued field $(\K,\V)$ is added as a
set of axioms.

The theory {\bf DOAG}\nobreak${}_\infty$ of divisible ordered Abelian
groups with last element $\infty$ admits quantifier elimination; hence
it is sufficient to eliminate the $F$-quantifiers from an $L$-formula
$\phi$: we obtain an $F$-quantifier free $L$-formula $\phi'$ (most of
the time, this formula has more $\Gamma$-quantifiers than $\phi$), and
we can conclude using the quantifier elimination of {\bf
DOAG}\nobreak${}_\infty$.

This strategy allows us to get a new algorithmic proof of
theorem \ref{belette}, which is the topic of the third section
of \cite{Wei}: {\it The theory {\bf ACVF}$(L)$ admits quantifier elimination}.

\subsection{Ultrametric Graphs} \label{subsecUG}
%
To prove theorem \ref{belette}, we will need a lemma about {\it ultrametric
graphs.} Let $\Gamma$ be the divisible ordered Abelian group
$\GKac$. A graph of vertices $p_1,\dots,p_n$ is a subset $G$ of
$\{p_1,\dots,p_n\}^2$ such that if $(p_i,p_j)\in G$, then
$(p_j,p_i)\in G$. If $(p_i,p_j)\in G$, then it is an \textit{edge} of
$G$. The graph will be called {\it complete} if every pair
$(p_i,p_j)$ is an edge.

We consider graphs labeled by elements of $\Gamma\cup\{\infty\}$: to
each edge $(p_i,p_j)$ we associate an element
$\varepsilon_{ij}\in\Gamma\cup\{\infty\}$, and we impose that
$\varepsilon_{ij}=\varepsilon_{ji}$. Such a graph is called {\it
ultrametric} if every triangle in it is an {\it ultrametric triangle},
that is, has two vertices labeled by the same element of $\Gamma$,
and the third one is labeled by a greater or equal element.  We can
put $\varepsilon_{ii}=\infty$ as a convention, so that degenerated
triangles are ultrametric.

If we define
%
$$t(\varepsilon_{ij},\varepsilon_{ik},\varepsilon_{jk})
\>:\Leftrightarrow\>
(\varepsilon_{ij}=\varepsilon_{ik}) \wedge (\varepsilon_{ij}\leq
\varepsilon_{jk})\>,$$
then
$$T(\varepsilon_{ij},\varepsilon_{ik},\varepsilon_{jk})
\>:\Leftrightarrow\>
t(\varepsilon_{ij},\varepsilon_{ik},\varepsilon_{jk})\vee
t(\varepsilon_{ik},\varepsilon_{jk},\varepsilon_{ij})\vee
t(\varepsilon_{jk},\varepsilon_{ij},\varepsilon_{ik})$$
is the formula asserting that $(p_i,p_j,p_k)$ is an ultrametric
triangle inside the graph~$G$.

The complete graph of vertices $p_1,\dots,p_n$ with edges labeled by
$\varepsilon_{ij}$ is ultrametric if the following formula is true:
$$\bigwedge_{i<j<k}
T(\varepsilon_{ij},\varepsilon_{ik},\varepsilon_{jk})\>.$$
In an algebraically closed valued field, let $a_1,\dots,a_n$ be fixed
elements. Let $\varepsilon_{ij}=v(a_i-a_j)$. Then the complete graph of
vertices $a_1,\dots,a_n$ and of edges $(a_i,a_j)$ labeled by
$\varepsilon_{ij}$ is ultrametric.
%
%
\begin{lemma}[Ultrametric graphs] \label{ultra-graph}
In any formal theory of valued fields implying that the residue field is
infinite,
the assertion
   \[\exists_F\, x\ \bigwedge_{i=1,\dots,n} v(x-a_i)=\beta_i\]
   is equivalent to the formula expressing that the complete graph of
   vertices $a_1,\dots,a_n$ and $x$, with edges $(a_i,x)$ labeled by
   $\beta_i$, is ultrametric. The triangles $(a_i,a_j,a_k)$ being
   ultrametric, this is equivalent to $\bigwedge_{i<j} T_{ij}$ where
   $T_{ij}$ is $T(\varepsilon_{ij},\beta_i,\beta_j)$.
\end{lemma}
\begin{proof}{Proof}\parindent0mm
Let $S_i(x)$ be the formula $v(x-a_i)=\beta_i$. We prove that
$$\left(\exists_F\, x\  \bigwedge_i S_i(x)\right)
\equi \bigwedge_{i<j} T_{ij}\>.$$
The implication $\Longrightarrow$ is clear.

For the reverse implication $\Longleftarrow$, we first note that
$$\left(T_{ij} \wedge (\beta_j<\beta_i)\right) \impl
\beta_j=\varepsilon_{ij}\>,$$
and that
$$\left(\beta_j=\varepsilon_{ij} \wedge (\beta_j<\beta_i) \wedge
S_i(x) \right) \impl S_j(x)\>.$$
Thus we have the following implication:
\begin{equation}\label{truc}
\left( T_{ij}\wedge (\beta_j<\beta_i) \wedge S_i(x)\right)
\impl S_j(x)\>.
\end{equation}
Hence we need to keep only those indices $i$ for which $\beta_i$ is
maximal among $\beta_1,\dots,\beta_n$. Let
$\beta=\max\{\beta_1,\dots,\beta_n\}$ and $I_1=\{ i\in\{1,\dots,n\}\,:\,
\beta_i=\beta \}$. Assume w.l.o.g. that $1\in I_1$.
We have
$$\bigwedge_{i<j} T_{ij} \wedge \bigwedge_{i\in I_1} S_i(x) \impl
\bigwedge_{1=1,\dots,n} S_i(x)$$
Note that for $i,j\in I_1$, $T_{ij}$ is equivalent to
$\varepsilon_{ij}\geq\beta$, and that $S_i(x)$ is the formula
$v(x-a_i)=\beta$.  We show that
$$\bigwedge_{i<j,\ i,j\in I_1} T_{ij} \impl \exists_F\, x \
\bigwedge_{i \in I_1} v(x-a_i)=\beta\>.$$
If $\beta=\infty$, we have $T_{ij}\impl (a_i=a_j)$ for all $i,j\in
I_1$, and in this case we take $x=a_i$ for any $i\in I_1$.
Now assume that $\beta<\infty$. If $\varepsilon_{ij}>\beta$, we obtain
$\left(S_i(x) \wedge T_{ij}\right) \impl S_j(x)$. We consider the
following case distinction:

\noindent$\bullet$ If $\varepsilon_{ij}>\beta$ for all $i,j\in I_1$ then
   $\left( \bigwedge_{i<j\ i,j\in I_1}T_{ij}\wedge
   S_1(x)\right)\impl\bigwedge_{i \in I_1} S_i(x)$. The formula
   $\exists_F\, x \ S_1(x)$ being always true, we have
   $\bigwedge_{i<j}T_{ij}\impl\exists_F\, x\ \bigwedge_i S_i(x)$.

\noindent$\bullet$ Else, we take in $I_1$ a subset $I_2$ which
   is maximal for the property that $\varepsilon_{ij}=\beta$ for all
   indices $i,j\in I_2$.  It suffices to show that
   $\exists_F\, x\bigwedge_{i\in I_2} S_i(x)$, since
   from the definition of $I_2$ we have
   $$\left( \bigwedge_{i<j,\ i,j\in I_1} T_{ij} \wedge \bigwedge_{i\in
   I_2}S_i(x) \right) \impl \bigwedge_{i\in I_1} S_i(x)\>.$$

\def\ov#1{\mathop{\mathrm{res}}{#1}}
We can assume w.l.o.g.\ that $1\in I_2$. We denote the natural map from
$\Vac$ to $\Vac/\MVac = \ResKac$ by $x\mapsto \ov{x}$.  We fix
$z\in\Kac$ such that $v(z)=\beta$. The field $\ResKac$ is infinite
since it is algebraically closed; thus we can choose $x\in\Kac$ such
that
%
%
$$\bigwedge_{i\in I_2}                       
\quad \ov{\left({x-a_1\over z}\right)} \neq
\ov{\left({a_i-a_1\over z}\right)}\>.$$
This $x$ verifies $v(x-a_i)=\beta$, for all $i\in I_2\,$.
This concludes the proof.
\end{proof}
\begin{remark}
\rm  We can give a geometric description of the set
$$S=\{x\in\Kac \mathrel{:} \bigwedge_{i=1,\dots,n} v(x-a_i)=\beta_i
\}\>.$$
   We use the notations of the proof. Set $C_\beta(a) =
   \{x\mathrel{:} v(x-a)=\beta\}$. We have
$$ S= \bigcap_{i=1,\dots,n} C_{\beta_i}(a_i) = \bigcap_{i\in I_1}
C_{\beta}(a_i)\>. $$

   If $\beta=\infty$, $S$ is reduced to one element in $K$. Now suppose
   $\beta<\infty$. If $\varepsilon_{ij}>\beta$ for all $i,j\in I_1$,
   then $S=C_{\beta_i}(a_i)$ for all $i\in I_1$. If for some $i,j\in
   I_1$, $\varepsilon_{ij}=\beta$, take $I_2$ as in the proof. We have
   $S=\bigcap_{i\in I_2} C_{\beta}(a_i)$.
   Suppose that $1\in I_2$. The set $C_{\beta}(a_1)$ is an infinite
   disjoint union of open disks $B^\circ_\beta(\zeta) =
   \{x\mathrel{:}v(x-\zeta)>\beta\}$, where $v(\zeta-a_1)=\beta$.
   There is a bijection between the disks $B^\circ_\beta(\zeta)$ and
   the residue field of $\Kac$, given by
$$B^\circ_\beta(\zeta) \mapsto f(\zeta)={\rm res}
                 \left({\zeta-a_1\over  z}\right)\>.
$$
   We have the following equality:
  $$S=\bigcap_{i\in I_2} C_{\beta}(a_i) = \bigcup_{\textstyle{
  v(\zeta-a_1)=\beta\atop \forall i\in I_2\setminus\{1\}\;f(\zeta)\neq
  f(a_i)} } B^\circ_\beta(\zeta)\>.$$
   This union is nonempty because there are infinitely many values
   possible for $f(\zeta)$, but only finitely many for $f(a_i)$.
\end{remark}
\begin{remark} \label{elim-q-lineaire}
\rm\parindent0mm
Another formulation of lemma \ref{ultra-graph} is that we have a quantifier
elimination for {\it linear formulas} in {\bf ACVF}$(L)$: given a
formula
$$\exists_F\, x\ \bigwedge_i v(x-x_i)=\beta_i$$
we put $\varepsilon_{ij}=v(x_i-x_j)$, and the above
formula is equivalent to
$$\bigwedge_{i<j} T(\varepsilon_{ij},\beta_i,\beta_j)\>.$$
%
\end{remark}

An easy consequence is the following lemma:

\begin{lemma} \label{Ultra-graph2}
Take any complete ultrametric graph of vertices
$p_1,\dots,p_n$, with edges labeled by
$\varepsilon_{ij}\in\Gamma\cup\{\infty\}$, and elements $x_1,\dots,x_l
\in\Kac$ (with $l<n$), such that $v(x_i-x_j)=\varepsilon_{ij}$ for all
$i,j\leq l$. Then there exist $x_{l+1},\dots,x_n\in \Kac$ such that
$v(x_i-x_j)=\varepsilon_{ij}$ for all $i,j$.
\end{lemma}
%
%
\subsection{Univariate existential decision procedure}\label{subsecUEDP}
We are going to prove that existential problems in a single variable
$x$ can be solved in
$(\Kac,\Vac)$.
\begin{definition} We define {\em univariate $F$-conditions} by
\begin{itemize}
   \item[\bf (i)] For any $P(X)\in\K[X]$, the condition
   $\Phi(x)\>:\Leftrightarrow\>P(x)=0$ is a univariate $F$-condition.
   \item[\bf (ii)] Take any $\gamma,\delta\in\GK$, $q,r\in
   \mathbb{Q}^{>0}$, and any $P(X),Q(X)\in\K[X]$. The condition
   $\Phi(x)\>:\Leftrightarrow\>
   v(P(x))+ q\cdot \gamma \mathrel{\Box} v(Q(x)) +
   r\cdot \delta$, where
   $\Box$ is either $=$ or $<$, is a univariate $F$-condition.
   \item[\bf (iii)] Take any $P(X)\in\K[X]$. The condition
   $\Phi(x)\>:\Leftrightarrow\>
   v(P(x)) < \infty$ is a univariate $F$-condition.
   \item[\bf (iv)] If $\Phi(x),\Psi(x)$ are univariate $F$-conditions,
   then $\Phi(x)\wedge\Psi(x)$ and $\Phi(x)\vee\Psi(x)$ are univariate
   $F$-conditions.
\end{itemize}
Conditions of the form {\bf (i)}, {\bf (ii)} and {\bf (iii)} are
called {\em atomic $F$-conditions}.
\end{definition}
\begin{definition} We define {\em $\Gamma$-conditions} by
\begin{itemize}
   \item[\bf (i)] For any $\delta\in\GK\cup\{\infty\}$,
   $q_1,\dots,q_n\in\mathbb{Q}^{>0}$, $r\in \{1,n \}$, the condition\\
   $\Phi(\bar a)\>:\Leftrightarrow\>
   q_1\cdot a_1+\cdots+q_r\cdot a_r \mathrel{\Box}
    q_{r+1}\cdot a_{r+1}+\cdots+q_n\cdot a_n + \delta$,\\
   where
   $\Box$ is either $=$, $>$ or $<$, is a $\Gamma$-condition on $\bar a$.
   \item[\bf(ii)] If $\Phi(\bar a),\Psi(\bar a)$ are
   $\Gamma$-conditions
   on $\bar a$, then so are $\Phi(\bar a)\wedge\Psi(\bar a)$
   and $\Phi(\bar a)\vee\Psi(\bar a)$.
\end{itemize}
Conditions of the form {\bf (i)} are called {\em atomic
$\Gamma$-conditions}.
\end{definition}
It is well known that such conditions are equivalent
to some condition of the following form, which is by definition
a \textit{disjunctive normal form}:
$$\bigvee_{i=1}^n \bigwedge_{j=1}^{m_i} \Phi_{ij}\>,$$
where the $\Phi_{ij}$ are atomic conditions. Moreover, given any
univariate condition $\Phi(x)$, there is an algorithm which computes
a disjunctive normal form for $\Phi(x)$.

We say that $\xi\in\Kac$ satisfies a univariate $F$-condition $\Phi(x)$
if $\Phi(\xi)$ holds in $\Kac$, and that
$\alpha_1,\dots,\alpha_n\in\GKac$ satisfy a $\Gamma$-condition
$\Phi(\bar a)$ if $\Phi(\bar\alpha)$ holds in $\GKac$.

\par\smallskip
We recall the following result without proof. See Theorem~5.6
in~\cite{CLR} or Corollary~3.1.17 in~\cite{Ma}.

\begin{proposition}[Existential Decision Procedure in DOAG${}_\infty$]
\mbox{ }\par\noindent
Let $\Phi(\bar a)$ be a $\Gamma$-con\-di\-tion. Then there is an algorithm
to decide whether there are some $\alpha_1,\dots,\alpha_n\in\GKac$
satisfying $\Phi(\bar a)$ or not. If the answer is yes, the algorithm
provides such a n-tuple. We call it a {\em witness} of the condition.
\end{proposition}

We now prove the following theorem:
%
\begin{theorem}[Univariate Existential Decision Procedure
in ACVF]\label{udp}
Let $\Phi(x)$ be a univariate condition. Then we have an algorithm to
decide whether there is some $\xi\in\Kac$ satisfying $\Phi(x)$ or
not. If the answer is yes, the algorithm gives a description of a
witness $\xi\in\Kac$ such that $\Phi(\xi)$ holds; the algorithm decides
whether $\xi$ is unique or not, and if this is the case then $\xi$ is
in $\Kh$.
\end{theorem}
%
%
\begin{proof}{Proof} We give an existential decision procedure for a conjunction
$$\Phi(x)\;:\;\bigwedge_{i=1}^n \Phi_i(x)$$
where the $\Phi_i$'s are atomic conditions. It suffices to use it
several times to obtain an existential decision procedure for a univariate
condition put in a disjunctive normal form, and hence for every
univariate condition.

\noindent$\bullet$\enspace{\bf First case:} One of the $\Phi_i(x)$
(let's say $\Phi_1(x)$) is of the form $P(x)=0$. Let $k=\deg P$, and
$\xi_1,\dots, \xi_k$ be
the roots of $P$. Let $Q_1(x),\dots,Q_r(x)\in
K[x]$ be the polynomials appearing in the other
$\Phi_i(x)$'s. We can
use {\bf SimVal} to obtain the multiset of $(r+1)$-tuples of values
$[(v(\xi_i),v(Q_1(\xi_i)),\dots,v(Q_r(\xi_i)))]_{i=1,\dots,k}$.

It suffices now to check, for each $(\nu,\nu_1,\dots,\nu_r)$ in this
list, whether the conditions $\Phi_1,\dots,\Phi_n$ are verified:
\begin{itemize}
\item for a $\Phi_k$ of the form $Q_i(x)=0$, test whether $\nu_i=\infty$,
\item for a $\Phi_k$ of the form
$ v(Q_i(x)) + q\cdot\gamma \mathrel{\Box} v(Q_j(x))+ r\cdot \delta$,
test whether
$\nu_i + q\cdot\gamma \mathrel{\Box} \nu_{j}+r\cdot \delta$
(where $\Box$ is either $=$, $>$ or $<$).
\item for a $\Phi_k$ of the form $v(Q_i(x)) < \infty$, test whether
$\nu_i < \infty$.
\end{itemize}
If there are no $(r+1)$-tuples in this multiset such that these
conditions are verified, then there is no $\xi\in\Kac$ satisfying
$\Phi(x)$; if there are $m\leq k$ of these multisets satisfying these
conditions, we know that $m$ of the roots of $P$ can be chosen for
$\xi$.

If $m=1$, then remark~\ref{rem Disc elem} shows that the corresponding
root of $P$ is in $\Kh$.

\noindent$\bullet$\enspace{\bf Second case:} Assume now that there is
no condition $\Phi_i(x)$ of the form $P(x)=0$ among the $\Phi_i(x)$.
For each $i$, let $P_i(x)$ and $Q_i(x)$ be the polynomials appearing
in atomic formulas $\Phi_i\,:\, v(P_i(x)) + q_i\cdot\gamma_i
\mathrel{\Box_i} v(Q_i(x)) + r_i\cdot \delta_i$ (where $\Box_i$ is
either $=$, $>$ or $<$), and $\Phi_i\;:\, v(P_i(x)) <
\infty$ (in that case, set $Q_i=1$, $q_i=r_1=1$, $\gamma_i=0$ and
$\delta_i=\infty$ for the sequel).


We construct the following formulas:
$$\displaylines{\Phi'(x,\bar c,\bar d) \,:\,
\left(\bigwedge_{i=1}^n v(P_i(x))= c_i \wedge
v(Q_i(x))= d_i \right)\cr
\Phi''(\bar c,\bar d)\,:\, \left(\bigwedge_{i=1}^n
  c_i+ q_i\cdot\gamma_i\mathrel{\Box_i} d_i + r_i\cdot
\delta_i\right)\>.\cr}$$
The variables $\bar c= c_1,\dots, c_n$ and
$\bar d= d_1,\dots, d_n$ stand for elements of
$\GKac$. We have
$$ \exists x\in\Kac\>\> \Phi(x) \equi \exists \bar c,\bar d
\in\GKac\>\exists x\in\Kac \>\> \Phi'(x,\bar c,\bar d)\wedge
\Phi''(\bar c,\bar d)\>.$$

Consider a problem of the following form:
$$\Psi(x,\bar b)\,:\,\exists x\in\Kac\>
\bigwedge_{i=1}^m v(R_i(x))= b_i \>, $$
where each $R_i(X)$ is a polynomial of $\K[x]$, and the $ b_i$'s
are indeterminates.

We introduce all the roots $r_1,\dots,r_N$ of the polynomials
$R_1,\dots,R_m$. We can compute $N$ with the algorithm {\bf
GraphRoots}, as well as the values $\varepsilon_{ij}=v(r_i-r_j)$, for
all $i,j$, and the multiplicity $\mu_{jk}$ of $r_k$ as a root of
$R_j$.  We have an equivalence
$$\Psi(x,\bar b)\equi \exists\, x\in\Kac\> \exists
   a_1\cdots a_N\in\GKac\>
  \bigwedge_{i=1}^N v(x-r_i)= a_i \wedge
  \Psi_1(\bar a,\bar b)\>,$$
where $\Psi_1$ is a conjuction of formulas of the form
$ b_j = \sum_{k} \mu_{jk}\cdot a_k$.

    From the ultrametric graph lemma we have
$$ \exists x\in\Kac\>
\bigwedge_{i=1}^N v(x-r_i)= a_i \equi\bigwedge_{i<j}
T(\varepsilon_{ij}, a_i, a_j)\>.$$
Hence we can write that $\Psi(x,\bar b)$ is equivalent to a
problem in $\GKac$\ :
$$\Psi(x,\bar b)\equi\exists\bar a\in\GKac\>\bigwedge_{i<j}
T(\varepsilon_{ij}, a_i, a_j) \wedge \Psi_1(\bar a,\bar b)\>.$$
Now we can do that for $\Psi=\Phi'$. We obtain that $\exists
x\in\Kac \> \Phi'(x,\bar c,\bar d)$ is equivalent to $\exists \bar
a\in\GKac\> \Phi'''(\bar a,\bar c,\bar d)$, where $\Phi'''(\bar a,\bar
c,\bar d)$ is a $\Gamma$-condition. We have proved
$$ \exists x\in\Kac\>\Phi(x) \equi \exists \bar a,\bar c, \bar d
\in\GKac \>\Phi'''(\bar a,\bar c,\bar d) \wedge \Phi''(\bar c, \bar
d)\>.$$
We can apply the existential decision procedure for {\bf
DOAG}\nobreak${}_\infty$ to this formula. If there is no solution,
then there is no $\xi\in\Kac$ satisfying $\Phi(x)$. If there is a
solution, we can use it together with lemma~\ref{ultra-graph} to
describe an element $\xi\in\Kac$ satisfying $\Phi(x)$. Of course,
there is no unicity in that case.
\end{proof}
\begin{remark}
{\rm The first case of our proof can in fact be treated as a
particular case of the second, replacing $P(x)=0$ by $v(P(x))=\infty$:
in that case the existential decision procedure in
{\bf DOAG}\nobreak${}_\infty$
will give $a_i=\infty$ for some $i$, and then $v(x -r_i)=\infty$
implies $\xi = r_i$. However, the proof is clearer with this
distinction. Moreover, it would be less easy to show that in the case
of unicity, the witness is in $\Kh$.}
\end{remark}

\subsection{Quantifier Elimination}\label{subsecQEL}
Quantifier elimination algorithms very often come from existential
decision procedures in the one variable case. If such a decision
procedure is ``uniform'' it can be performed ``with parameters''. This
gives a good way for eliminating the quantifier in a formula with only
one existential quantifier.  For the real algebraic case see,
e.g., \cite{BCR} chapter 1.  In the present section, we will treat
the case of algebraically closed valued fields.
%
%
\begin{definition} Take $n\in\mathbb{N}$, and denote by $\bar y$ an
n-tuple $(y_1,\dots,y_n)$ of $F$-variables. Let $C_1(\bar y),
\dots,C_m(\bar{y})$ be atomic $L$-formulas with $y_1,\dots,y_n$ as
the only free variables.\\
   \noindent{\bf 1.} We say that $\bigvee_i C_i(\bar y)$ is a {\em finite
   exclusive disjunction} if
$$ \forall_F \bar{y}\quad \bigvee_{i=1}^m C_i(\bar{y})
\>\>\wedge\>\>
\bigwedge_{i\neq j} \neg C_i(\bar{y})\vee\neg C_j(\bar{y})$$
   holds. In that case we write
$$\C_i=\left\{\bar y\in\K^n\mathrel{:} C_i(\bar y)\right\}\>.$$
   Then $K^n$ is the disjoint union of $\C_1,\dots,\C_m$. The family
   $\C_i$ is a {\em definable partition of the space $\K^n$}. Note that
   we allow that some $\C_i$ may be empty.\\
   \noindent{\bf 2.} Let $D_{ij}(\bar y)$, for $i=1,\dots,m$ and
   $j=1,\dots,\ell_i$, be atomic $L$-formulas such that $\bigvee_{ij}
   D_{ij}(\bar y)$ is a finite exclusive disjunction. We say that
   $\bigvee_{ij} D_{ij}$ is a refinement of $\bigvee_i C_i$ if for all
   $i$, we have
$$C_i(\bar y) \equi \bigvee_{j=1}^{\ell_i} D_{ij}(\bar y)\>,$$
   or, equivalently
$$\C_i = \bigcup_{j=1}^{\ell_i} \D_{ij}\>,$$
   where $\D_{ij}=\left\{\bar y\in\K^n\mathrel{:} D_{ij}(\bar
   y)\right\}$. Note that this union is a disjoint union.
\end{definition}
%
%
We denote by $\bar{Y}$ an $n$-tuple of indeterminates
$Y_1,\dots,Y_n$. The ring $\K\left[\bar{Y}\right]$ is $\K[Y_1,\dots,Y_n]$.
%
%
We can apply the algorithms given in the previous
section to polynomials with parameters. Consider
$P(\bar{Y},X)\in\K\left[\bar{Y},X\right]$ as a polynomial in $X$
with parameters $\bar{Y}$.
\begin{proposition}[Algorithms with parameters] $\;\;$ \\
{\bf 1.} The \NPA applied to $P$ provides
   \begin{itemize}
       \item[\bf (i)] a finite exclusive disjunction $\bigvee_i
       C_i(\bar y)$,
       \item[\bf (ii)] for each $i$, an integer $k_i$ and a multiset
       $[t_1(\bar{y}),\dots, t_{k_i}(\bar{y})]$, where each
       $t_j(\bar{y})$ is an $L_\Gamma$-term,
    \end{itemize}
    such that for all $\bar{y}\in\C_i$, $k_i=\deg_X P(\bar{y},X)$, and
    if $[\xi_1,\dots \xi_{k_i}]$ denotes the multiset of roots of
    $P(\bar{y},X)$, then $[t_1(\bar{y}),\dots, t_{k_i}(\bar{y})]$ is
    $[v(\xi_1),\dots,v(\xi_{k_i})]$. In other words, in each case of
    the above exclusive disjunction, the algorithm computes the values
    of the roots of $P(\bar{y},X)$.\\
    \noindent{\bf 2.} Keep the notation of the previous statement. Let
    $Q_1,\dots,Q_r\in\K\left[\bar{Y},X\right]$ be polynomials in $X$
    with parameters $\bar{Y}$. The algorithm {\bf SimVal} applied to
    $P,Q_1,\dots,Q_r$ provides
    \begin{itemize}
       \item[\bf (i)] a refinement $\bigvee_{ij} D_{ij}(\bar y)$ of
       $\bigvee_i C_i(\bar y)$,
       \item[\bf (ii)] for each case $i,j$ (with
       $j\in\{1,\dots,\ell_i\}$) a multiset of $(r+1)$-tuples of
       $L_\Gamma$-terms $\left[\left(t_s(\bar y), u_s^1(\bar
       y),\dots,u_s^r(\bar y)\right)\right]_{s=1,\dots,k_i}$,
   \end{itemize}
    such that for all $\bar y\in\D_{ij}$, if $[\xi_1,\dots \xi_{k_i}]$ is
    the multiset of roots of $P(\bar{y},X)$, then $\left[\left(t_s(\bar
    y), u_s^1(\bar y), \dots, u_s^r(\bar y) \right)
    \right]_{s=1,\dots,{k_i}}$ is $\left[\left(v(\xi_s), v(Q_1(\xi_s),
    \dots, v(Q_r(\xi_s)) \right)\right]_{s=1,\dots,{k_i}}$.\\
   \noindent{\bf 3.} Take $P_1,\dots,P_s\in\K\left[\bar Y,X\right]$. The
   algorithm {\bf GraphRoots} applied to $P_1,\dots,P_s$ provides
   \begin{itemize}
       \item[\bf (i)] a finite exclusive disjunction $\bigvee_i
       C_i(\bar y)$,
       \item[\bf (ii)] for each $i$, an integer $N_i$ and a finite
       family\\
       $\left((\mu_{jk})_{j\in[1,s],k\in[1,N_i]},
       (t_{k,\ell}(\bar y))_{1\leq j<\ell\leq N}\right)$,  where the
       $\mu_{jk}$ are integers and the $t_{k,\ell}(\bar y)$ are
       $L_\Gamma$-terms,
    \end{itemize}
     such that for all $\bar y\in\C_i$, $N_i$ is the number of roots
     of $P_1\cdot\dots\cdot P_s$, and for some ordering
     $(\xi_1,\ldots,\xi_{N_i})$ of these roots, $\mu_{jk}$ is the
     multiplicity of $\xi_k$ as a root of $P_j$, and $t_{k,\ell}(\bar
     y)$ is $v(\xi_k - \xi_\ell)$.
\end{proposition}
\begin{proof}{Proof} For the first statement, write $P(\bar Y,X) =
    q_n(\bar y)\cdot X^n + \cdots + q_0(\bar y)$.  Consider the
    exclusive disjunction
$$\Bigl(q_0(\bar y)=\ldots=q_n(\bar y)=0\Bigr)\>\vee\>\bigvee_{i=0}^n
\left(v(q_i(\bar y))<\infty\;\wedge \bigwedge_{j=i+1}^n q_{n-j}(\bar
y)=0\right)\>.$$
    In each case of this disjunction the degree in $X$ of $P(\bar y,X)$
    is fixed. We are going to refine it to obtain the desired
    disjunction. Apply the \NPA in any fixed case of this disjunction:
    its result depends naturally on a new disjunction, each case of it
    expressing a different shape for the Newton Polygon of $P$. More
    precisely, if $m>0$ is $\deg_X P(\bar y,X)$, for each $\ell\leq
    m+1$ and each $\ell$-tuple $(k_1,\dots,k_\ell)$ of non-negative
    integers such that $0=k_1<\cdots<k_\ell=m$, we can write a formula
    $C_{m,\ell,k_1,\dots,k_\ell}(\bar y)$ expressing that
    $(k_1,v(q_{k_1}(\bar y))),\dots,(k_\ell,v(q_{k_ell}(\bar y)))$ are
    the consecutive vertices of the Newton Polygon of $P$. In each
    fixed case $C_{m,\ell,k_1,\dots,k_\ell}$, the values of the roots
    are the $L_\Gamma$-terms ${1\over k_{i+1}-k_i}(v(q_{k_i}(\bar
    y))-v(q_{k_{i+1}}(\bar y)))$.

\ms {\bf Example:}
    Set $R(\bar{Y},X)=a(\bar{Y})X^2 + b(\bar{Y})X + c(\bar{Y})$; we
    omit the parameters $\bar{Y}$ in the sequel: $a$ stands for $a(\bar
    Y)$, and so on.  \begin{itemize}
    \item If $v(a)<\infty$, and $2v(b)\geq v(a) + v(c)$, then
     $\xi_1,\xi_2\in\Kac$, the roots of $R$ considered as a polynomial
     in $X$, both have value ${1\over2} (v(c) - v(a))$.
    \item If $v(a)<\infty$, and $2v(b) < v(a) + v(c)$, then there is one
    root of value $v(b) - v(a)$ and the other of value $v(c) - v(b)$.
    \item If $a=0$ and $v(b)<\infty$, then there is a single root, of
    value $v(c)-v(b)$.
    \item If $a=0$ and $b=0$ and $v(c)<\infty$, then there is no root.
    \item If $a=b=c=0$, then $\forall x\in\Kac$, $R(\bar{y},x)=0$.
\end{itemize}

    Now we turn to the second statement. The algorithm {\bf SimVal}
    applies the \NPA to $P$: this is our first disjunction. Then it
    computes some Tschirnhaus transformation of $P$. The degree of
    $P$ being fixed in each case of the disjunction, this can be done
    without refining it. The results of this computations
    are new polynomials in $K\left[\bar Y,X\right]$.  We apply the \NPA
    to each of these polynomials, after refining the disjunction. We
    obtain some lists of $L_\Gamma$-terms, from which we can construct
    the list we want, under a few conditions to eliminate ``bad
    coincidences'' (cf. \ref{comput sim val}); these conditions give
    rise to a new refinement of the disjunction.

    For the third statement, just note that {\bf GraphRoots} uses
    {\bf SimVal} iteratedly; then the result comes from the second
    statement.
\end{proof}

Now we are able to prove theorem~\ref{belette}.
\begin{proof}{Proof of theorem \ref{belette}}
   We recall that there are classical and easy arguments
   (\cite{Wei}) showing that it suffices to eliminate an
   $F$-quantifier $\exists_F\, x$ in a formula such as
   $\exists_F\, x\ \bigwedge_{k=1,\dots,n} \Phi_k(\bar{y},x)$, where
   each $\Phi_k(\bar{y},x)$ is either an atomic $F$-formula like
   $P(\bar{y},x)=0$ with $P(\bar y,x)\in\mathbb{Z}[\bar y,x]$, or an
   atomic $\Gamma$-formula. Note that an atomic $F$-formula
   $P(\bar y,x)\not=0$ can be replaced by the $\Gamma$-formula
   $v(P(\bar y,x))<\infty$.
   So we are done if we prove the following proposition.
\end{proof}
\begin{proposition} \label{propbelette}
   There is an algorithmic procedure that computes, from a formula
   $\exists_F\,x\ \bigwedge_{k=1,\dots,n} \Phi_k(\bar{y},x)$ (where
   each $\Phi_k(\bar{y},x)$ is either an atomic $F$-formula like
   $P(\bar{y},x)=0$ with $P(\bar y,x)\in\mathbb{Z}[\bar y,x]$, or an
   atomic $\Gamma$-formula), an equivalent quantifier free formula
   $\Psi(\bar{y})$.
\end{proposition}

A geometric form of this proposition is the following
(for the real algebraic case see, e.g., theorem 2.2.1 of \cite{BCR}).
Let $\K$ be a subfield of $\L$. A {\em basic $v$-constructible set
defined over $\K$} in $\L^n$ is a set of the form $\left\{\bar{x}\in
\L^n\mathrel{:}\Phi(\bar{x}) \right\}$ where $\Phi(\bar{x})$ is either
an atomic $F$-formula like $P(\bar{x})=0$ with $P(\bar x)\in \K[\bar
x]$, or an atomic $\Gamma$-formula (which is built by using only
constants in $\K$ and $v(\K)$). A {\em $v$-constructible set
defined over $\K$} in $\L^n$ is any boolean combination of basic
$v$-constructible sets defined over $\K$.
\begin{proposition} \label{belettegeometric}
   Let $\L$ be an algebraically closed valued field, and $\K$ a subfield.
   Then the image $\pi(S)$ of a $v$-constructible set $S$ defined over
   $\K$ under the canonical projection from $\L^n$ onto $\L^{n-1}$ is
   again a $v$-constructible set defined over $\K$.
   Moreover, there is an algorithmic procedure that uses only
   computations inside $\K$ to get a description of $\pi(S)$ from a
   description of $S$.
\end{proposition}
\begin{proof}{Proof of proposition \ref{propbelette}}

We can apply our univariate decision procedure (theorem~\ref{udp}) with
parameters in order to eliminate $x$. This procedure uses {\bf SimVal}
and {\bf GraphRoots} with parameters: it will provide an exclusive
disjunction $\bigvee_i C_i(\bar{y})$, and in each case of this
exclusive disjunction, a formula $\Psi_i(\bar y)$ without
$F$-quantifiers (but perhaps with some new $\Gamma$-quantifiers if for
$\bar y\in\C_i$ we are in the second case of the proof of~\ref{udp})
such that
$$\forall \bar y \in \C_i,\> \exists_F\, x\bigwedge_{k=1,\dots,n}
\Phi_k(\bar{y},x) \equi \Psi_i(\bar y)\>.$$
Thus we have
$$ \exists_F\, x\bigwedge_{k=1,\dots,n} \Phi_k(\bar{y},x) \equi
\bigvee_i C_i(\bar y) \wedge \Psi_i(\bar y)\>.$$
This concludes the proof.
\end{proof}

\begin{remark}
{\rm
The strategy used in \cite{Wei} was first to give an
elimination for linear formulas, and then a procedure which decreases
the degrees of polynomials. There was no geometric idea at first sight,
although there may be a geometric content hidden in the proof. We
believe that the two procedures are in fact different.}
\end{remark}

When we use this quantifier elimination with the theory
{\bf ACVF}$(\K,\V)$  we get as a particular case
a decision
procedure for a closed formula with coefficients in a valued field
$\K$ given as in the introduction.
\begin{theorem} Take a formula
$$
\Theta(\bar y)\;:\;
Q^1_F x_1\dots Q^n_F x_n\ \Phi(\bar \alpha,\bar y,\bar x)
$$
where each $Q^i_F$ is $\forall_F$ or $\exists_F$
and $\bar \alpha=\alpha_1,\dots,\alpha_m$ are elements of $\K$.
We have an algorithm for computing a quantifier free formula $\Psi(\bar y)$
equivalent  to $\Theta(\bar y)$.
As a particular case, when $\bar y$ is the empty sequence, we can
decide whether the formula $\Theta(\bar y)$ is true in $\Kac$ or not.
Moreover, if the
formula is purely existential, i.e., $Q^1_F,\dots, Q^n_F$ are
existential quantifiers $\exists_F$, then the algorithm provides a
witness $\bar{\xi}\in(\Kac)^n$ such that $\Phi(\bar \xi)$ is true. If we
have a result of unicity such as
$$\forall_F\bar{x},\bar{y}\ (\Phi(\bar\alpha,\bar{x})
\wedge \left(\Phi(\bar\alpha,\bar{y}) \impl\bar{x}=\bar{y}\right)\>,$$
then this witness is in $(\Kh)^n$.
\end{theorem}
\begin{proof}{Proof}
Let us explain how we get the test point.
We apply the quantifier elimination procedure to
$$Q^1_F x_1\dots Q^n_F x_n\ \Phi(\bar a,\bar x)$$
obtained after replacement of each $\alpha_i$ by a new indeterminate
$a_i$. The result is a quantifier-free formula $\Psi(\bar a)$, such that
$$Q^1_F x_1\dots Q^n_F x_n\ \Phi(\bar a,\bar x) \equi \Psi(\bar a)\>.$$
It suffices to test whether $\Psi(\bar \alpha)$ is true or not.

If all quantifiers $Q^i_F$ are existential, we can find formulas
$\Psi_{k}(\bar a,x_1,\dots,x_k)$ for $k=1$ to $n-1$, such that
$$
\begin{array}{cc}
&\exists_F\, x_1\dots\exists_F\, x_n\  \Phi(\bar a,x_1,\dots,x_n)\cr
\equi&\exists_F\, x_1\dots\exists_F\, x_{n-1}\ \Psi_{n-1}(\bar
a,x_1,\dots,x_{n-1})\cr
\vdots&\vdots\cr
\equi&\exists_F\, x_1\ \Psi_1 (\bar a,x_1)\cr
\equi&\Psi(\bar a)\cr
\end{array}
$$

If $\Psi(\bar\alpha)$ is true and we apply the decision procedure of
theorem~\ref{udp} to the sentence $\exists_F\, x_1\ \Psi_1
(\bar\alpha,x_1)$,
we find $\xi_1\in\Kac$ such that $\Psi_1 (\bar\alpha,\xi_1)$ holds. We
apply again the decision procedure to $\exists_F\, x_2\ \Psi_2
(\bar\alpha,\xi_1,x_2)$ and we find $\xi_2\in\Kac$ such that $\Psi_2
(\bar\alpha,\xi_1,\xi_2)$ holds, and so on. In this way, we find
$\xi_1,\dots,\xi_n\in\Kac$ such that
$\Phi(\bar\alpha,\xi_1,\dots,\xi_n)$ holds.

If the $n$-tuple $(\xi_1,\dots,\xi_n)$ satisfying
$\Phi(\bar\alpha,x_1,\dots,x_n)$ is unique, then $\xi_1$ satisfying
$\Psi_1(\bar\alpha ,x_1)$ is unique and theorem~\ref{udp} shows that
$\xi_1\in\Kh$. Repeating this argument $n$ times, we conclude that, in
this case, $\xi_1,\dots,\xi_n\in\Kh$.
\end{proof}

%
%





\addcontentsline{toc}{section}{References}



\tableofcontents


\begin{thebibliography}{50}


\bibitem{BCR} J.~Bochnak, M.~Coste, M.-F.~Roy:
{\it G\'eom\'etrie alg\'ebrique r\'eelle}, Springer (1987)

\bibitem{CLR} M.~Coste, H.~Lombardi, M.-F.~Roy:
{\it Dynamical method in algebra: Effective Nullstellens\"atze},
Annals of Pure and Applied Logic {\bf 111} (2001), 203--256

\bibitem{DD} J.~Della Dora, C.~Dicrescenzo, D.~Duval:
{\it About a new method for computing in algebraic number fields},
in: Proceedings Eurocal'85, Springer Lecture Notes in
Computer Science {\bf 204} (1985), 289--290

\bibitem{DD2} C.~Dicrescenzo, D.~Duval: {\it Algebraic extensions and
algebraic closure in Scratchpad}, in: Symbolic and algebraic computation
(ISSAC 88), Springer Lecture Notes in Computer Science {\bf 358} (1989),
440--446

\bibitem{KL} F.-V.~Kuhlmann, H.~Lombardi:
{\it Construction du hens\'elis\'e d'un corps valu\'e},
Journal of Algebra {\bf 228} (2000), 624--632

\bibitem{lapin} H.~Perdry: {\it Aspects constructifs de la th\'eorie
des corps valu\'es}, Th\`ese de doctorat en Math\'ematiques et
Applications de l'Universit\'e de Franche-Comt\'e (2001)

\bibitem{Ma} D.~Marker: {\it Model Theory: An Introduction}, Graduate
Texts in Mathematics {\bf 217}, Springer (2002)

\bibitem{Wei} V.~Weispfenning: {\it Quantifier elimination and
decision procedure for valued fields}, in: Models and sets,
Springer Lecture Notes in Math.\ {\bf 1103} (1984), 419--472

\end{thebibliography}
\end{document}